\documentclass{amsproc}

\usepackage{amsmath, amssymb, amsthm}
\usepackage[latin1]{inputenc}
\usepackage{geometry}
\usepackage[normalem]{ulem}
\usepackage{cancel}
\usepackage[bookmarks=false,colorlinks=true,linkcolor=black,citecolor=black,filecolor=black,urlcolor=black]{hyperref}

\newtheorem{trm}{Theorem}

\newtheorem{lem}[trm]{Lemma}

\newcommand{\Z}{\ensuremath{\mathbb{Z}}}

\newcommand{\z}{\ensuremath{\mathcal{Z}}}
\newcommand{\GEN}[1]{\langle #1 \rangle}
\newcommand{\Exp}{\textnormal{Exp}}

\newcommand{\inv}{^{-1}}
\newcommand{\matriz}[1]{\begin{array} #1 \end{array}}

\date{}
\title{Group rings with Lie metabelian set of symmetric elements}

\thanks{The first author has been partially supported by FAPEMIG and CAPES (Proc nº BEX4147/13-8) of Brazil. The second and third authors has been partially supported by Ministerio de Econom\'{\i}a y Competitividad project MTM2012-35240 and Fondos FEDER and Fundaci\'{o}n S\'{e}neca of Murcia 04555/GERM/06. }

\author{Osnel Broche}
\address{Osnel Broche, Departamento de Ci\^{e}ncias Exatas, Universidade Federal de Lavras,
Caixa Postal 3037, 37200-000, Lavras, Brazil }
\email{osnel@dex.ufla.br}

\author{\'{A}ngel del R\'{\i}o}
\address{\'{A}ngel del R\'{\i}o, Departamento de Matem\'{a}ticas, Universidad de Murcia,
30100, Murcia, Spain}
\email{adelrio@um.es}

\author{Manuel Ruiz}
\address{Manuel Ruiz, Departamento de M\'{e}todos Cuantitativos e
Inform\'{a}ticos, Universidad Polit\'{e}cnica de Cartagena, C/ Real, sn   30.201 Cartagena, Spain}
\email{manuel.ruiz@upct.es}

\begin{document}
\begin{abstract}
Let $R$ be a commutative ring of characteristic zero and $G$ an arbitrary group.  In the present paper we classify  the groups $G$ for which the set of symmetric elements  with respect to the classical involution of the group ring $RG$ is Lie metabelian.
\end{abstract}

\maketitle

\section{Introduction}

If $*$ is an involution on a ring $R$ then the symmetric (respectively, anti-symmetric) elements of $R$ with respect to $*$ are the elements of 
$R^+=\{r\in R|\; r^*=r\}$ (respectively, $R^-=\{r\in R|\; r^*=-r\}$). 
It is well known that crucial information of the algebraic structure of $R$ can be determined by that of $R^+$. 
An important results of this nature is due to Amitsur who proved that for an arbitrary algebra $A$ with an involution $*$, if $A^+$ satisfies a polynomial identity then $A$ satisfies a polynomial identity \cite{A}.

Let $R$ be a commutative ring, let $G$ be a group and let $*$ be a group involution extended by linearity to the group ring $RG$.
More precisely, if $r=\sum_{g\in G} r_g g \in RG$ with $r_g\in R$ for each $g\in G$ then   $r^*=\sum_{g\in G} r_g g^*$.
During the last years many authors have paid attention to the algebraic properties of the symmetric and anti-symmetric elements of $RG$ and in particular to their Lie properties. 
The characterization of the Lie nilpotence of $RG^+$ and $RG^-$, for the case of the classical involution, was given in three different papers when $R$ is a field of characteristic different from 2. In the
first paper  \cite{GS}  it was considered the
case in which the group $G$  has no 2-elements. The case in which the group $G$ has $2$-elements was solved in
\cite{L} for $RG^+$ and in \cite{GS1} for $RG^-$. For group involutions extended by linearity to the whole group ring the Lie nilpotence of $RG^+$ was study in \cite{GPS1} when the group $G$ has no 2-elements, and completed for an arbitrary group $G$ in \cite{LSS2009JPAAA}. The case of oriented involutions was study in \cite{JHP}.

A particular case of Lie nilpotence, the commutativity, has been studied for $RG^+$ and $RG^-$. For the classical involution, this was study in \cite{Broche2006} for $RG^+$ and in \cite{BrochePolcino2007} for $RG^-$. For the case of an arbitrary group involution extended by linearity to the whole group ring, the commutativity of $RG^+$ was studied in \cite{JR} and that of $RG^-$ in \cite{JR2,BJPR}. This was extended to oriented involutions in \cite{BP2} for $RG^+$ and in \cite{BJR} for $RG^-$. Finally, for nonlinear involutions the commutativity of $RG^+$ was studied in \cite{R} and that of $RG^-$ in \cite{R2}.

The Lie solvability of $RG^+$ and $RG^-$ was studied in \cite{LSS2009Forum} for the classical involution when the group $G$ has no $2$-elements. The question of when $RG^+$ is Lie metabelian has been studied under some conditions, namely for $G$ a finite group of odd order without elements of order 3 and the classical involution in \cite{LR}; for $R$ a field of characteristic different from $2$ and $3$ and $G$ a periodic group without $2$-elements and an arbitrary group involution extended by linearity to $RG$ in \cite{CLS2013}. In all theses cases $RG^+$ is Lie metabelian if and only if $G$ is abelian. Finally, in \cite{CLS2013}, it is also shown that for $G$ a finite group of odd order and $R$ a field of characteristic different from 2, if $RG^+$ is Lie metabelian then $G$ is nilpotent.


The goal of this paper is to characterize the group rings $RG$, with $R$ a commutative ring of characteristic zero and $G$ an arbitrary group, for which the set $RG^+$ of symmetric elements with respect to the classical involution is Lie metabelian.  More precisely we prove the following result.

\begin{trm}\label{Main}
Let $R$ be a commutative ring of characteristic 0 and $G$ a group. Denote by $RG^+$ the set of symmetric elements of $RG$ for the classical involution.
Then $RG^+$ is Lie metabelian if and only if one of the following conditions holds:
\begin{enumerate}
 \item $\GEN{g\in G:g^2\ne 1}$ is abelian.
 \item $G$ contains an elementary abelian subgroup of index 2.
 \item $G$ has an abelian subgroup $B$ of index 2 and an element $x$ of order 4 such that $b^x=b^{-1}$ for every $b\in B$.
 \item The center of $G$ is $\{g\in G:g^2=1\}$ and it has index 4 in $G$.
\end{enumerate}
\end{trm}

\section{Preliminaries and notation}
In this section we introduce the basic notation and definitions.
The centre of $G$ is denoted $\z(G)$ and its exponent is denoted by $\Exp(G)$. If $g$ is a group element of finite order we will denote by $\circ (g)$ its order. If $g,h\in G$ then $g^h=h\inv g h$ and $(g,h)=g^{-1}h^{-1}gh$.

For elements $a$ and $b$ in an arbitrary ring we use the standard notation for the additive commutator: $[a,b]=ab-ba$, also known as Lie bracket.
Recall that a ring $R$ is called Lie metabelian if $[[a,b],[c,d]]=0$ for all $a,b,c,d\in R$.
More generally we say that a subset $X$ of $R$ is Lie metabelian if the same identity holds for all the elements of $X$.
We also say that $X$ is commutative if the elements of $X$ commute.

More generally,   if $X$ and $Y$ are subsets of a ring then $[X,Y]$ denotes the additive subgroup generated by the Lie brackets $[x,y]$ with $x\in X$ and $y\in Y$. Observe that $X$ is Lie metabelian if and only if $[X,X]$ is commutative.

The following subsets of $RG$ play an important role:
\begin{eqnarray*}
X^+&=&\{g+g^{-1}:g \in G, g^2\neq 1\}\cup\{g:g\in G, g^2=1\} \\
\breve{G}&=&\{g-g^{-1}:g\in G\}
\end{eqnarray*}
Note that $RG^+$ is generated as an $R$-module by $X^+$ and
therefore $RG^+$ is Lie metabelian if and only if so is $X^+$. In particular,  if $RG^+$ is commutative then obviously $RG^+$ is Lie metabelian. The nonabelian groups $G$ satisfying that $RG^+$ is commutative have been classified in \cite{Broche2006}. These groups are precisely the Hamiltonian 2-groups, and are included in (2) and (3) of Theorem~\ref{Main}.
Therefore in the rest of the paper we will assume that $RG^+$ is not commutative.

Also, 	\begin{equation}\label{RGRGX}
	[RG^+,RG^+]\subseteq R\breve{G}
	\end{equation}
where $R\breve{G}$ denotes the $R$-submodule of $RG$ generated by $\breve{G}$. In fact, to see this it is enough to consider  $g,h,x,y\in G$ with $x^2=y^2=1$ and write the Lie brackets of generators of $RG^+$ in the following form:
\begin{eqnarray*}
 {}[g+g\inv,h+h\inv]&=&gh-(gh)\inv+gh\inv-(gh\inv)\inv+g\inv h - (g\inv h)\inv +  (hg)\inv - hg, \\
 {}[g+g\inv,x] &=& gx-(gx)\inv+g\inv x - (g\inv x)\inv, \\
 {}[x,y]&=&xy-(xy)\inv
\end{eqnarray*}

We will need the following result.

\begin{trm}\label{CasiAntisimetricos}\cite{BrochePolcino2007}
Let $R$ be a commutative ring with unity with characteristic 0 and let $G$ be any group.
Then $\breve{G}$ is commutative if and only if one of the following conditions holds:
\begin{enumerate}
\item $K=\GEN{g\in G: g^2\ne 1}$ is abelian (and thus  $G=K\rtimes \GEN{x}$  with $x^2=1$ and $k^x=k\inv$ for all $k\in K$).
\item $G$ contains an elementary abelian subgroup of index $2$.
\end{enumerate}
\end{trm}

\section{Sufficiency condition}
In this section we prove the sufficiency part of Theorem~\ref{Main}.
Assume first that $G$ satisfies either condition (1) or condition (2) of Theorem~\ref{Main}.
Then $\breve{G}$ is commutative by Theorem~\ref{CasiAntisimetricos} and therefore $[RG^+,RG^+]$ is commutative by (\ref{RGRGX}).
Thus $RG^+$ is Lie metabeliano.

Secondly suppose that $G$, $B$ and $x$ satisfy condition (3) of Theorem~\ref{Main}. Then clearly all the elements of $G\setminus B$ have order 4 and the elements of $G$ of order 2 are central. Therefore in order to prove that $RG^+$ is Lie metabelian it is enough to show that the following set is commutative
    $$C= \{ [ g+g^{-1}, h+h^{-1}] :g,h\in G,\; g^2\neq 1\neq h^2 \}.$$
As $B$ is abelian, $[ g+g^{-1}, h+h^{-1}]=0$ for elements $g,h\in B$.
If $g\in B$ and $h\in G\setminus B$ then $g^h=g^{-1}$, so  $(g+g^{-1} )h=h(g+g^{-1})$. Therefore, $[g+g^{-1},h+h^{-1}]=0$.
Finally, if $g,h\in G\setminus B$ then $h=bg$ for some $b\in B$.
Since $b^g=b^{-1}$ and $g^2=g^{-2}$, we have that
$$[ g+g^{-1}, h+h^{-1}]=(g+g^{-1})b(g+g^{-1})-b(g+g^{-1})^2=2(b^{-1}-b)(1+g^2).$$
Hence, $C \subseteq \{ 2(b^{-1}-b)(1+g^2):\ b\in B,\; g\in G\setminus B \}\cup \{0\} \subseteq RB$ and thus $C$ is commutative as desired.

Finally, assume that $G$ satisfies condition (4) of Theorem~\ref{Main}. As in the previous case the elements of order 2 are central and hence it is enough to prove that the elements of $C$ commute.
Notice that $G/\z(G)$ has exactly three non-trivial cyclic subgroups, say $\GEN{x\z(G)}$, $\GEN{y\z(G)}$ and $\GEN{z\z(G)}$, and $z=uxy$ for some $u\in \z(G)$.
Moreover $G'=\{1,t=(x,y)\}$ and
    $$(x+x^{-1})(y+y^{-1})(z+z^{-1}) = (x^2y^2+t)(1+x^2)(1+y^2)u.$$
and therefore
    \begin{equation}\label{Suficiency31}
     (1-t)(x+x^{-1})(y+y^{-1})(z+z^{-1}) = 0.
    \end{equation}

Consider $x_1,y_1,x_2,y_2\in G$ with $[x_1+x_1^{-1},y_1+y_1^{-1}]\ne 0$ and $[x_2+x_2^{-1},y_2+y_2^{-1}]\ne 0$.
We have to show that $[[x_1+x_1^{-1},y_1+y_1^{-1}],[x_2+x_2^{-1},y_2+y_2^{-1}]]=0$.
Observe that $y_1\not \in \GEN{\z(G),x_1}$ and $x_1\not \in \GEN{\z(G),y_1}$ and therefore $G=\GEN{\z(G),x_1,y_1}$.
Similarly $G=\GEN{\z(G),x_2,y_2}$.
Moreover $t=(x_1,y_1)=(x_2,y_2)$ and $(y_i+y_i^{-1})(x_i+x_i^{-1})=t(x_i+x_i^{-1})(y_i+y_i^{-1})$ for $i=1,2$, so that
    \begin{equation}\label{Suficiency32}
     [x_i+x_i^{-1},y_i+y_i^{-1}] = (1-t)(x_i+x_i^{-1})(y_i+y_i^{-1}).
    \end{equation}
As $G/\z(G)$ has exactly 3 non-trivial cyclic subgroups then either $x_2\z(G)\in \{x_1\z(G),y_1\z(G)\}$ or $y_2\z(G)\in \{x_1\z(G),y_1\z(G)\}$.
By symmetry, one may assume that $x_2\z(G)=x_1\z(G)$. So $x_2=ux_1$ with $u\in \z(G)$. Moreover either $y_2\in y_1\z(G)$ or $y_2\in x_1y_1\z(G)$.
In the first case $y_2=vy_1$ with $v\in \z(G)$ and
	\begin{eqnarray*}
	 [x_2+x_2^{-1},y_2+y_2^{-1}]&=&(1-t)(ux_1+ux_1^{-1})(vy_1+vy_1^{-1})=(1-t)uv(x_1+x_1^{-1})(y_1+y_1^{-1})\\&=&uv[x_1+x_1^{-1},y_1+y_1^{-1}].
	\end{eqnarray*}
In the second case $\GEN{x_1\z(G)}=\GEN{x_2\z(G)}$, $\GEN{y_1\z(G)}$ and $\GEN{y_2\z(G)}$ are the three different cyclic subgroups of $G/\z(G)$.
Then using (\ref{Suficiency31}) and (\ref{Suficiency32}) and the fact that $t$ is central we have
    \begin{eqnarray*}
    &&[x_1+x_1^{-1},y_1+y_1^{-1}][x_2+x_2^{-1},y_2+y_2^{-1}] = \\
    &&(1-t)(x_1+x_1^{-1})(y_1+y_1^{-1})(1-t)(x_2+x_2^{-1})(y_2+y_2^{-1}) = 0\\
    \end{eqnarray*}
and
    \begin{eqnarray*}
    &&[x_2+x_2^{-1},y_2+y_2^{-1}][x_1+x_1^{-1},y_1+y_1^{-1}] = \\
    &&(1-t)(x_2+x_2^{-1})(y_2+y_2^{-1}) (1-t)(x_1+x_1^{-1})(y_1+y_1^{-1})= 0\\
    \end{eqnarray*}
In both cases $[[x_1+x_1^{-1},y_1+y_1^{-1}],[x_2+x_2^{-1},y_2+y_2^{-1}]]=0$, as desired.
\section{Necessary condition}
Now we assume that $RG^+$ is Lie metabelian and we will prove that $G$ satisfies one of the conditions (1)-(4) of Theorem~\ref{Main}. This is easy if $\breve{G}$ is commutative, by Theorem~\ref{CasiAntisimetricos}. Thus unless otherwise stated we assume that $RG^+$ is Lie metabelian and $\breve{G}=\{g-g^{-1}:g\in G\}$ is not commutative.

A relevant role in the proof is played by the following normal subgroups of $G$:
	\begin{eqnarray*}
	 A&=&\GEN{g\in G : g^2 = 1}, \\
   B&=&\GEN{g\in G : \circ(g)\ne 4 }
	\end{eqnarray*}

\bigskip
\noindent {\bf 4.1} \underline{Properties of $A$}
\medskip

\bigskip
The first lemmas address the properties of $A$. (In fact the first lemma does not use the assumption that $\breve{G}$ is not commutative.)

\begin{lem}\label{PropiedadesA}
\begin{enumerate}
\item\label{ElementsA} Every element of $A$ is of the form $ab$ with $a^2=b^2=1$.
\item\label{ACasiAntisimetrico} $\breve{A}=\{a-a^{-1}:a\in A\}$ is commutative.
\item\label{(x,A)=1ox^2inA} For every $x\in G$ either $(x,A)=1$ or $x^2\in A$.
\end{enumerate}
\end{lem}

\begin{proof}
(\ref{ElementsA}) We have to prove that the product of elements of order 2 of $G$ is also the product of at most two elements of order 2.
By induction it is enough to show that if $x_1,x_2,x_3$ are elements of order 2 in $G$ then $x_1x_2x_3$ is the product of at most two elements of order 2. This is clear if $(x_1,x_2)=1$ or $(x_2,x_3)=1$. So assume that $(x_1,x_2)\ne 1\ne (x_2,x_3)$. If $(x_1,x_3)=1$ then $x_2x_3x_1$ is the product of two elements of order 2 and it is conjugate of $x_1x_2x_3$. Thus we may also assume that $(x_1,x_3)\ne 1$. As, by hypothesis, $RG^+$ is Lie metabelian and $x_1^2=x_2^2=x_3^2=1$, we have
\begin{eqnarray*}
0=[[x_1,x_2],[x_2,x_3]]&=&x_1x_3+x_2x_1x_3x_2+x_2x_3x_2x_1+x_3x_2x_1x_2\\&&-(x_3x_1+x_1x_2x_3x_2+x_2x_1x_2x_3+x_2x_3x_1x_2)
\end{eqnarray*}
Then $x_1x_3$ is one of the elements of the negative part and by assumption it is not any of the first three summands. Therefore $x_1x_3=x_2x_3x_1x_2$ and hence
$(x_1x_2x_3)^2 = x_1 (x_2x_3x_1x_2) x_3 = x_1^2 x_3^2=1$, as desired.

(\ref{ACasiAntisimetrico}) Let $a\in A$ with $a^2\ne 1$. By (\ref{ElementsA}), $a=xy$ with $x^2=y^2=1$. Then $a-a^{-1}=xy-yx=[x,y]$.
Now, using that $RG^+$ is Lie metabelian we deduce that $\{a-a\inv:a\in A\}$ is commutative.

(\ref{(x,A)=1ox^2inA}) Assume that $(x,A)\ne 1$. Then $(x,a)\ne 1$ for some $a\in G$ of order 2. By assumption
\begin{eqnarray*}
0&=&[[x+x^{-1},a],[xa+ax^{-1},a]] \\&=&
2(ax^{-2}+x^{-2}a+xax+axaxa-x^2a-ax^2-x^{-1}ax^{-1}-ax^{-1}ax^{-1}a).
\end{eqnarray*}
Then $x^2a\in \{ax^{-2},x^{-2}a,xax,axaxa\}$. However $x^2a$ is not one of the last two elements because $(x,a)\ne 1$. Therefore $x^2a=ax^{-2}$ or $x^2=x^{-2}$. In the first case $x^2a$ has order 2 and therefore it belongs to $A$. In the second case $x^4=1$. In both cases $x^2\in A$, as desired.
\end{proof}

\begin{lem}\label{Aabelian}
$A$ is abelian.
\end{lem}

\begin{proof}
Recall that we are assuming that $\breve{G}$ is not commutative and $A$ is not abelian. By statement (\ref{ACasiAntisimetrico}) of Lemma~\ref{PropiedadesA}, $\breve{A}$ is commutative. Then $G\ne A$ and $A$ satisfies one of the two conditions of Theorem~\ref{CasiAntisimetricos}.
In both cases $A$ contains elements $b$ and $c$ with $b^2\ne 1=c^2$ and $(b,c)\ne 1$ and $(b^2,c)=1$. Moreover, if $x\in G\setminus A$ then either $x$, $xb$, $xc$ or $xbc$ do not commute with neither $b$ nor $c$. Thus we may assume that $(x,b)\ne 1 \ne (x,c)$.

As in the proof of statement (\ref{ACasiAntisimetrico}) of Lemma~\ref{PropiedadesA}, $b-b\inv$ is the additive commutator of two elements of order 2.
Then
\begin{eqnarray*}
0&=&[b-b^{-1},[x+x^{-1},c]] \\
 &=& bxc+b^{-1}cx^{-1}+b^{-1}cx+bx^{-1}c+cx^{-1}b+cxb+x^{-1}cb^{-1}+xcb^{-1} \\
 & & -bcx-bcx^{-1}-b^{-1}xc-b^{-1}x^{-1}c  -xcb-x^{-1}cb -cxb^{-1} -cx^{-1}b^{-1}
\end{eqnarray*}
and therefore
$$
 bxc\in \{bcx,bcx^{-1}, b^{-1}xc, b^{-1}x^{-1}c, xcb, x^{-1}cb, cxb^{-1}, cx^{-1}b^{-1}\}
$$
If $bxc=bcx$ then $(c,x)=1$, a contradiction. If $bxc=bcx^{-1}$ it follows that $(xc)^2=1$ and thus $x\in A$, a contradiction. If $bxc=b^{-1}xc$ then $b^2=1$, a contradiction. If $bxc=cx^{-1}b^{-1}$ then $(bxc)^2=1$ and hence $x\in A$, a contradiction. Therefore only four possibilities remains:
\begin{itemize}
\item[(a.1)] $bxc=b^{-1}x^{-1}c$ and thus $b^2x^2=1$.
\item[(a.2)] $bxc=xcb$.
\item[(a.3)] $bxc=cxb^{-1}$
\item[(a.4)] $bxc=x^{-1}cb$.
\end{itemize}

Using similar arguments for $[b-b^{-1},[xc+cx^{-1},c]]=0$ we get that
$$
 bx\in \{bcxc,bx^{-1}, b^{-1}x, cx^{-1}cb, xb, x^{-1}b^{-1}, b^{-1}cx^{-1}c, cxcb^{-1}\}
$$
If $bx=bcxc$ then $(x,c)=1$, a contradiction. If $bx=bx^{-1}$ then $x^2=1$ and thus $x\in A$, again a contradiction. If $bx=b^{-1}x$ then $b^2=1$, a contradiction.  If $bx=x^{-1}b^{-1}$, then $(bx)^2=1$ and thus $x\in A$, again a contradiction. Therefore, since $(b,x)\neq 1$ only three possibilities remains, namely
\begin{itemize}
 \item[(b.1)] $bxb^{-1}=cx^{-1}c$
 \item[(b.2)] $b^2(xc)^2=1 $
 \item[(b.3)] $bxb=cxc$
 \end{itemize}

We now consider the two cases for $A$ mentioned at the beginning of the proof.

\emph{Case 1}: $K=\GEN{g\in A\;:g^2\ne 1}$ is abelian and $A=K\rtimes \GEN{\alpha}_2$ with $k^\alpha=k\inv$ for every $k\in K$.
In particular $b^c=b\inv$.

If (b.1) holds $bxb^{-1}=cx^{-1}c$ then $(cbx)^2=(cb)^2=1$ and thus $x\in A$, a contradiction.

Assume now that (b.2)  $b^2(xc)^2=1 $ holds.
Then $bxc=b^{-1}cx^{-1}$. If case (a.1) holds then  $bx=b^{-1}x^{-1}$ and therefore it follows that $b^{-1}cx^{-1}=bxc=b^{-1}x^{-1}c$ and thus $(c,x^{-1})=1$, a contradiction. If case (a.2) holds then $b^{-1}cx^{-1}=bxc=xcb$ and  $(xcb)^2=xcbb^{-1}cx^{-1}=1$ and thus $x\in A$ a contradiction. If case (a.3) holds then $cbx^{-1}=b^{-1}cx^{-1}=bxc=cxb^{-1}$ and so $bx^{-1}=xb^{-1}$ which implies that $(xb^{-1})^2=1$ and hence $x\in A$, a contradiction. If (a.4) $bxc=x^{-1}cb=x^{-1}b^{-1}c$ holds then $bx=x^{-1}b^{-1}$ and therefore $bx\in A$ again a contradiction.

Finally assume that (b.3) $bxb=cxc$ holds. Then $bx=cxcb^{-1}=cxbc$ and thus $bxc=cxb$. If case (a.1) holds then  $bx=b^{-1}x^{-1}$ and therefore it follows that $(xcb)^2=xcbbxc=xcbb^{-1}x^{-1}c=1$ and thus $x\in A$, a contradiction. If case (a.2) holds then $cxb=bxc=xcb$ and thus $(x,c)=1$, a contradiction. If case (a.3) holds then $cxb=bxc=cxb^{-1}$ and hence $b=b^{-1}$, a contradiction. If (a.4) holds using the same argument as in case (b.2) we get a contradiction,  that finishes the proof in this case.

\emph{Case 2}. $A$ contains an elementary abelian subgroup $C$ of index 2 in $G$. We can assume that $c\in C$, $b^2\in C$ and $b$ has order 4.

Assume first that (b.1) $bxb^{-1}=cx^{-1}c$ holds. Then $(cbx)^2=(cb)^2=(bc)^2$ and thus $(bc)^2\neq 1\neq (cb)^2$, because $cbx\not\in A$.
If case (a.1) holds then
 \begin{eqnarray*}
[[b+b\inv,c],[c,x+x\inv]] &=&
4(bx+b^{-1}x+cbxc+cb^{-1}xc-b^{-1}cxc-bcxc-cbcx-cb^{-1}cx)\\
&=& 4(bx+b^{-1}x+cbxc+cb^{-1}xc-xb-x^{-1}b-cbcx-cb^{-1}cx)\\
&=&0.
\end{eqnarray*}
Then $bx\in \{xb,x^{-1}b,cbcx, cb^{-1}cx\}$. Since $(b,x)\ne 1$ and $(b,c)\ne 1$ it follows that $bx=cb^{-1}cx$. Then $cb=b^{-1}c$ and hence $(cb)^2=1$ a contradiction.

If case (a.2) holds then
\begin{eqnarray*}
[[b+b\inv,c],[c,x+x\inv]] &=&
4(cb^{-1}xc+cb^{-1}x^{-1}c+cbxc+cbx^{-1}c\\
&&-b^{-1}cxc-bcxc-b^{-1}cx^{-1}c-bcx^{-1}c)\\
&=&
4(cb^{-1}xc+b^{-1}x+cbxc+bx\\
&&-x^{-1}b^{-1}-bcxc-xb^{-1}-bcx^{-1}c)\\
&=&0.
\end{eqnarray*}
Thus
$bx\in \{x^{-1}b^{-1},bcxc,xb^{-1}, bcx^{-1}c\}$. Since $(bx)^2\ne 1\ne (c,x)$ and $(cx)^2\ne 1$ it follows that $bx=xb^{-1}$. Then using (a.2) we get that $b^{-1}c=cb$ and therefore $(cb)^2=1$ a contradiction.

If case (a.3) holds then $x^2=(bc)^2=(cb)^2$ and thus $(b,x^2)=1=(x^2,c)$.  Hence we have that
  \begin{eqnarray*}
&&[[b+b\inv,c],[b+b\inv,cx+(cx)\inv]] =\\
&&  2(3x^{-1} + 2b^2x^{-1} + 2x^{-1}b^2 + 2bx^{-1}b + 2b^{-1}x^{-1}b^{-1} + 3bx^{-1}b^{-1} + b^{-1}x^{-1}b + b^2x^{-1}b^2\\
  &&-3x - 2b^2x - 2xb^2 - 2bxb - 2b^{-1}xb^{-1} - 3bxb^{-1} - b^{-1}xb - b^2xb^2)= 0,
\end{eqnarray*}
and hence $x^{-1}\in\{x, b^2x, xb^2, bxb, b^{-1}xb^{-1}, bxb^{-1}, b^{-1}xb, b^2xb^2\}$. Notice that $b^2\neq x^2$ because in this case $b^2=(bc)^2$ and therefore $(b,c)=1$, a contradiction. Thus, since $x^2\neq 1$, $b^2\neq x^2$,  $(bx)^2\neq 1 \neq (xb^{-1})^2$ it follows that $x^{-1}=b^2xb^2$. Then $(b^2x)^2=1$ and hence $b^2x\in A$ a contradiction.

Finally, if case (a.4) holds then $bxc=x^{-1}cb=cbx$ and $(bx)^2=(bc)^2=(cb)^2=(xb)^2$.
Thus the following commutator
\begin{eqnarray*}
 0&=& [[c,b+b\inv],[c,x+x\inv]] \\
&=&
xb+xb^{-1}+x^{-1}b+x^{-1}b^{-1}+cbcx+cbcx^{-1}+cb^{-1}cx+cb^{-1}cx^{-1}+bcxc\\
&&+bcx^{-1}c+b^{-1}cxc+b^{-1}cx^{-1}c +cxbc+cxb^{-1}c+cx^{-1}bc+cx^{-1}b^{-1}c\\
&& -bx-bx^{-1}-b^{-1}x-b^{-1}x^{-1}-cbxc-cbx^{-1}c-cb^{-1}xc  -cb^{-1}x^{-1}c\\
&&-cxcb-cxcb^{-1}-cx^{-1}cb-cx^{-1}cb^{-1}-xcbc-xcb^{-1}c-x^{-1}cbc-x^{-1}cb^{-1}c\\
&=&
2(xb+xb^{-1}+x^{-1}b+2x^{-1}b^{-1}+b^2xb+b^2xb^{-1}+b^2x^{-1}b^{-1}\\
&& -2bx-b^{-1}x-bx^{-1}-b^{-1}x^{-1}-bxb^2-bx^{-1}b^2-b^{-1}x^{-1}b^2).
\end{eqnarray*}
Moreover $xb\not\in \{bx,b^{-1}x^{-1},bx^{-1}b^2\}$ because $(b,x)\ne 1$, $xb\not\in A$ and $xb^{-1}\not\in A$. Therefore
$xb \in \{bx^{-1},b^{-1}x,bxb^2,b^{-1}x^{-1}b^2\}.$
Hence either $x^b=x^{-1}$ or $b^x=b^{-1}$.
But in the first case $x^{-1}bc=bxc=x^{-1}cb$, in contradiction with $(b,c)\ne 1$.
Thus $b^x=b^{-1}$ and taking the following commutator we have
\begin{eqnarray*}
 0&=& [[b+b\inv,c],[b+b\inv,cx+cx\inv]] =16(cb+cb^{-1}-bc+b^{-1}c)x^{-1}cb^{-1}.
\end{eqnarray*}
Hence $cb=b^{-1}c$ and thus $(bx)^2=(bc)^2=1$ in contradiction with $bx\not\in A$.

Secondly assume that $(b.2)$ $b^2(xc)^2=1$ holds. In case (a.1) we have $x^2=b^2=(xc)^2$ and therefore $x=cxc$ in contradiction with $(x,c)\ne 1$. In case (a.2), $(bxc)^2=b(xc)^2b=1$, in contradiction with $bxc\not\in A$. For cases (a.3) and (a.4)  we consider the following double commutator
\begin{eqnarray*}
 0 &=& [[c,b+b\inv],[c,cx+(cx)\inv]] \\
   &=& 4(bxc +b^{-1}xc+cbcx^{-1}c+cb^{-1}cx^{-1}c            -bx^{-1}c-b^{-1}x^{-1}c-cbcxc-cb^{-1}cxc)
\end{eqnarray*}
Then $bxc\in \{bx^{-1}c,b^{-1}x^{-1}c,cbcxc,cb^{-1}cxc\}$.
However $bxc \not\in \{bx\inv c,cbcxc\}$, since $x^2\ne 1\ne (b,c)$.
Thus either $b^2=x^2$ or $b^c=b\inv$. If $b^2=x^2$ by (b.2) we have that $x^2=(xc)^2$ and hence $(x,c)=1$, a contradiction.
Thus $b^c=b\inv$. In case (a.3) since $b^2=(xc)^2$ and $b^c=b^{-1}$ we get that $cx^{-1}=(xc)^3=bcxb^{-1}$ and thus $x^{-1}=b^{-1}xb^{-1}$. Thus $(bx^{-1})^2=1$  yielding to a contradiction with $x\in A$. In case (a.4) we get that $bx=x^{-1}cbc=x^{-1}b^{-1}$ and hence $(bx)^2=1$, again a contradiction

Finally assume that (b.3) $bxb=cxc$ holds.

In case (a.1) $x^2=b^2$. Now consider the commutator
\begin{eqnarray*}
[[b+b\inv, c],[b+b\inv,x+x\inv]] &=& 8(cbxb^{-1}+cb^{-1}xb^{-1}+bcbx^{-1}+b^{-1}cbx^{-1}\\&&-cx^{-1}-cb^2x^{-1}-bcxb^{-1}-b^{-1}cxb^{-1}) =0.
\end{eqnarray*}

Then it follows that $cbxb^{-1}\in\{-cx^{-1},cb^2x^{-1},bcxb^{-1},b^{-1}cxb^{-1}\}$. If $cbxb^{-1}=cx^{-1}$ then $bxb^{-1}=x^{-1}$ and since $x^2=b^2$ we get that $bxb=x$. Therefore by (b.3) we obtain that $cxc=bxb=x$ and hence $(c,x)=1$, a contradiction. If $cbxb^{-1}=cb^2x^{-1}=cx$ it follows that $(b,x)=1$, a contradiction. If $cbxb^{-1}=bcxb^{-1}$ then $(b,c)=1$, a contradiction. Finally if $cbxb^{-1}=b^{-1}cxb^{-1}$ then $cb=b^{-1}c$ and hence $(bcx)^2=bcxccbcx=bbxbcbcx=b^2x^2=1$, again a contradiction.

In case (a.2) notice that we have $(b^2,x)=1$.
In fact $b^2xc=bxcb=xcb^2=xb^2c$.
Now we consider
\begin{eqnarray*}
0&=& [[c,b+b\inv],[c,x+x\inv]] \\ &=&
4(bcx+b^{-1}cx+bcx^{-1}+b^{-1}cx^{-1}  -cbx-cb^{-1}x-cbx^{-1}-cb^{-1}x^{-1})c
\end{eqnarray*}
Then $bc \in \{cb\inv, cbx^{-2}, cb\inv x^{-2}\}$, because $bc\ne cb$.
But if $bc=cb\inv$ then $bxc=xb^{-1}c$ and so $bx=xb^{-1}$.
Therefore  $cxc=bxb=x$ and thus $(c,x)=1$, a contradiction.
Hence $x^{-2} \in \{(b,c),(bc)^2\}\in C$ and thus $x$ has order 4.
But if $x^2=(b,c)$, since $(b^2,x)=1$ then $(bcx)^2= bcx(cbx^2)x=bcxcbx^3=b^2xb^2x^3=1$ and so $bcx\in A$, a contradiction.
On the other hand,  if $x^2=(bc)^2$, by (a.2) we have that $bx=xcbc$ and then $cxc=bxb=x(cb)^2=x^{-1}$ yielding to a contradiction with $(cx)^2=1$.

In case (a.3) we have that  $bxc=cxb^{-1}$. Recall that by (b.3) we have that $cx=bxbc$ and therefore $bxc=cxb^{-1}=bxbcb^{-1}$, in contradiction with $(b,c)\ne 1$.

Finally assume that (a.4) $bxc=x^{-1}cb$ holds. Furthermore, from the relations $c^2=b^4=(c,b^2)=1$, $cxc=bxb$ we get the following computation

\begin{eqnarray*}
 0&=& [[c,b+b\inv],[c,x+x\inv]] \\
&=&
4xb+2xb^{-1}+2x^{-1}b+4x^{-1}b^{-1}+2b^2x^{-1}b^{-1}+2b^2xb\\
&& -4bx-2bx^{-1}-2b^{-1}x-4b^{-1}x^{-1}-2bxb^2-2b^{-1}x^{-1}b^2
\end{eqnarray*}

Moreover $xb\not\in \{bx,b^{-1}x^{-1},bxb^2\}$ because $(b,x)\ne 1$, $bx\not\in A$ and $xb\ne cxcb=bxb^2$. Thus
$xb \in \{bxb^2,b^{-1}x^{-1}b^2\}.$ If $xb=bxb^2$ it follows that $x=bxb=cxc$, in contradiction with $(b,c)\ne 1$. Thus $xb=b^{-1}x^{-1}b^2$ and hence $bx=x^{-1}b$. Therefore by (a.4) we get that $x^{-1}bc=bxc=x^{-1}cb$ yielding to a contradiction with $(b,c)\neq 1$. This finishes the proof of the lemma.
\end{proof}

\begin{lem}\label{AAbelConsecuencias}
Let $a\in A$ and $x,y\in G$ be such that $(a,x)\ne 1 \ne (a,y)$. Then
\begin{enumerate}
\item\label{Orden4} $x$ and $y$ have order 4.
\item\label{CuadradoCentral} $(x^2,y)=(x,y^2)=1$.
\item\label{Traza1} $aa^xa^ya^{xy}=1$.
\item\label{GModuloAAbeliano} $(x,y)\in A$.
\end{enumerate}
\end{lem}

\begin{proof}
(\ref{Orden4}) is a consequence of statement~(\ref{(x,A)=1ox^2inA}) of Lemma~\ref{PropiedadesA} and the fact that $A$ is abelian.

(\ref{CuadradoCentral}) Consider
$$\matriz{{l}
[[x+x^{-1},y+y^{-1}],[(xy)+(xy)^{-1},y+y^{-1}]] = \\
2 y^{-1} + 3xy^{-1}x  + xyxy^2 + y^{-1}xy^{-1}xy + y^{-1}xyxy^{-1} + 2yxyxy + 2y^{-1}xy^2x  \\ 
+y^{-1}xyx^{-1}y +  x^{-1}y^{-1}xy^2 + yx^{-1}y^{-1}xy + yxyx^{-1}y^{-1}  + 2x^2y^{-1}   \\ 
 + x^{-1}yxy^2       
+ y^{-1}xyx^{-1}y^{-1} 
+ yx^{-1}yxy           
+ yxyx^{-1}y           \\
 + 3xyx
+ xy^{-1}xy^2
+ 2y^{-1}xyxy
+ yxy^{-1}xy
+ yxyxy^{-1}
\\ 
+ 2y^{-1}xy^2x^{-1}\\
- 2y
- 3x^{-1}y^{-1}x^{-1}
- 3x^{-1}yx^{-1}
- 2yx^2
- 2xy^2xy
- 2x^{-1}y^2xy \\
- 2y^{-1}x^{-1}y^{-1}x^{-1}y^{-1}
- 2y^{-1}x^{-1}y^{-1}x^{-1}y
- y^{-1}x^{-1}y^{-1}xy^{-1}
- y^{-1}x^{-1}yx^{-1}y^{-1} \\
- y^{-1}x^{-1}yx^{-1}y
- y^{-1}x^{-1}yxy^{-1}
- y^{-1}xy^{-1}x^{-1}y^{-1}
- y^{-1}xy^{-1}x^{-1}y
- y^2x^{-1}y^{-1}x^{-1}
- y^2x^{-1}y^{-1}x \\
- y^2x^{-1}yx^{-1}
- y^2x^{-1}yx
- yx^{-1}y^{-1}x^{-1}y^{-1}
- yx^{-1}y^{-1}x^{-1}y
- yxy^{-1}x^{-1}y^{-1}
- yxy^{-1}x^{-1}y
}$$

By assumption the result of the previous calculation should be 0. Notice that if $xy$ and $xy^{-1}$ belong to $A$ then since $A$ is abelian $(x^2,y)=(y^2,x)=1$ as desired. So assume that $xy$ and $xy^{-1}$ do not belong to $A$.
Also we have that $y$ should appear in the support of the positive part. It cannot be one of the elements of the first line (after the equality) because neither $y,xy^{-1},xy$ nor $xy^2$ have order 2, as they do not belong to $A$.
If $y$ belongs to the support of the second line then $(x,y)=1$ or $x^2=y^2$, as desired.
If $y$ belongs to the support of the third line then $y^x=y^{-1}$ and if $y$ belongs to the fourth line then $x^y=x^{-1}$. In both cases $(x,y)=1$, as desired. Finally if $y=y^{-1}xy^2x^{-1}$ then $(x,y^2)=1$. Since $x$ and $y$ play symmetric rolls it also follows that $(x^2,y)=1$.

(\ref{Traza1}) By (\ref{CuadradoCentral}), both $x^2$ and $(ax)^2=aa^xx^2$ commute with $y$ and hence $a^ya^{xy}=(aa^x)^y=aa^x$. Thus $aa^xa^ya^{xy}=1$, as desired.

(\ref{GModuloAAbeliano}) By means of contradiction assume that $(x,y)\not\in A$. Using (2) we have $1\ne (x,y)^2 = (x^2y^2 (xy)^2)^2 = (xy)^4$.

\textbf{Claim}: $x^2=y^2$. By means of contradiction assume $x^2\ne y^2$ and consider the following double commutator.
$$\matriz{{l}
[[x+x^{-1},xy+(xy)^{-1}],[x+x^{-1},(xy)^2x+((xy)^2x)^{-1}]] = \\
8(y^{-1}+x^2y^{-1}+(xy)^2y+xyx^{-1}y^2+(yx)^2y^{-1}+y(xy)^2x^2y^2+(xy)^4y+(xy)^{3}x^{-1}y^2 \\
-y-yx^2-(yx)^2y-xyx-xyx^{-1}-(yx)^{3}x-(xy)^{3}x-(xy)^{3}x^{-1}).}$$
By assumption this is 0. Having in mind that $y\ne y^{-1}$, $x^2\ne y^2$,  and $(xy)^4\ne 1$, (in particular $(xy)^2\ne 1$, $(yx)^2y^2\ne 1\ne(xy)^2x^2y^2$), and comparing $y$ with the elements with positive coefficient we deduce that either $xyx^{-1}y=1$ or $(xy)^4=x^2$. In the first cases $(x,y)=y^2\in A$, yielding a contradiction. We conclude that $(xy)^4=x^2$ and hence $(xy)^4=(yx)^4$. Claiming symmetry we deduce that $(yx)^4 = y^2$. Then $y^2=(yx)^4=(xy)^4 =x^2$, again a contradiction. This finishes the proof of the claim.

Let $z=xy$. Then $z^2=(x,y)\not\in A$.
Consider the following double commutator
$$\matriz{{l}
[[x+x^{-1},a],[x+x^{-1},xz+(xz)^{-1}]] = \\
8(azx+azx^{-1}+a^xz^{-1}x^{-1}+a^xz^{-1}x-az^{-1}x^{-1}-az^{-1}x-a^xzx^{-1}-a^xzx)=0
}$$
Then $azx\in\{az^{-1}x^{-1}, az^{-1}x, a^xzx^{-1}, a^xzx\}$. If $azx=az^{-1}x^{-1}$ then $z^2=x^{-2}=x^2$ and thus $z^4=1$, a contradiction. If $azx=az^{-1}x$ then $z^2=1$, a contradiction. If $azx=a^xzx^{-1}$ then $a=a^xx^2=x^{-1}ax^{-1}$ and therefore $(ax)^2=1$, a contradiction because $x\not\in A$. Finally $azx=a^xzx$ and hence $a=a^x$ again a contradiction because $(a,x)\neq 1$.

\end{proof}

\begin{lem}\label{ACentral}
$A\subseteq \z(G)$.
\end{lem}

\begin{proof}Recall that we are assuming that $\breve{G}$ is not commutative and $A$ is not contained in the center of $G$. Let $x\in G$ and $a\in A$ be such that $(x,a)\ne 1$.
By Lemma~\ref{Aabelian}, $A$ is an elementary abelian 2-group. We claim that $G/A$ is elementary abelian 2-group too.
Indeed, if $y\in G$ then either $y$ or $yx$ does not commute with $a$. Thus $\GEN{A,x,y}$ with $x$ and $y$ satisfying the conditions of Lemma~\ref{AAbelConsecuencias}. Then $x^2,y^2,(x,y)\in A$ and therefore $\GEN{A,x,y}/A$ is elementary abelian. Thus $y^2\in A$. This proves the claim. 

By Theorem~\ref{CasiAntisimetricos}, $[G:A]>2$, for otherwise $\breve{G}$ is commutative. Hence $G\ne \GEN{A,x}$. Let $y\in G\setminus \GEN{A,x}$. By replacing $y$ by $xy$ if needed one may assume that $(a,y)\ne 1$. By replacing $y$ by $ay$ we may also assume that $t=(x,y)\ne 1$.
By Lemma~\ref{AAbelConsecuencias}, we deduce that $x$ and $y$ have order 4, $(x^2,y)=(x,y^2)=1=aa^xa^ya^{xy}$ and $t$ has order 2. Using this we deduce that $(t,x)=(t,y)=1$ and $(y+y^{-1})(x+x^{-1})=t(x+x^{-1})(y+y^{-1})$.

Thus
\begin{eqnarray}\label{axay}
\nonumber &&[[a,x+x^{-1}],[a,y+y^{-1}]] \\
 &&\hspace{3cm} = (aa^x+aa^y+aa^{xy}t+t-aa^{xy}-1-aa^yt-aa^xt)(1+x^2+y^2+x^2y^2)xy
\end{eqnarray}

By assumption this is 0 and therefore one of the following conditions hold:
$$aa^x\in \GEN{x^2,y^2}, \; aa^{y}\in \GEN{x^2,y^2}, \; t\in \GEN{x^2,y^2}, \; aa^{xy}t\in \GEN{x^2,y^2}.$$
This implies that one out of 16 equalities holds. However seven of them yields a contradiction with the fact that $aa^x$, $aa^y$, $(ax)^2$, $(ay)^2$, $(axy)^2$, $(xy)^2$ and $t$ are all different from 1. We classify the remaining nine equalities as follows:
$$\matriz{{lrclllrcl}
	\text{(a)} & aa^{xy}t &=&1. \\
	\text{(b)} & t &=&x^2.        & & \text{(b')} & t&=&y^2. \\
	\text{(c)} & aa^x&=&y^2.     & & \text{(c')} & aa^y&=&x^2. \\
	\text{(d)} & aa^x &=& x^2y^2.& & \text{(d')} & aa^y&=&x^2y^2.\\
	\text{(e)} & aa^{xy}t &=&x^2. & & \text{(e')} & aa^{xy}t&=&y^2. \\

}$$
By symmetry we only have to consider the five cases in the left column. On the other hand one can pass from case (e) to case (b) by replacing $x$ by $x_1=ax$.
Indeed, if case (e) holds then $(x_1,y)= a^x a^{xy}t = aa^xx^2 = x_1^2\ne 1$ . Therefore, we only have to consider cases (a)-(d).

Replacing in (\ref{axay}), $y$ by $ay$ and therefore $t$ by $(x,ay)=aa^xt$, and $y^2$ by $(ay)^2 = aa^y y^2$ we obtain
\begin{eqnarray}\label{axaay}
\nonumber &&[[a,x+x^{-1}],[a,ay+(ay)^{-1}]] \\
&&\hspace{4cm} = (aa^x+aa^y-aa^{xy}-1)(1+t)(1+x^2+aa^yy^2+aa^yx^2y^2)xay
\end{eqnarray}

In cases (a) or (b), equation (\ref{axaay}) takes the form
$$[[a,x+x^{-1}],[a,ay+(ay)^{-1}]] = 2(aa^x+aa^y-aa^{xy}-1)(1+x^2+aa^yy^2+aa^yx^2y^2)xay$$
Thus, in these cases $\GEN{x^2,aa^yy^2}$ contains either $aa^x$ or $aa^y$.
 This yields eight possible equalities, but again we can exclude five of them because neither $y^2$, $aa^x$, $aa^y$ nor $(ax)^2$ is 1. Thus
\begin{equation}\label{abc3}
 \text{In cases (a) or (b) either } x^2=aa^y,  y^2=a^xa^y, x^2y^2=a^xa^y \text{ or } x^2=y^2.
\end{equation}

On the other hand, in cases (c) and (d) equation (\ref{axaay}) takes the form
\begin{equation}\label{axaaycd}
[[a,x+x^{-1}],[a,(ay)+(ay)^{-1}]]
= 2(aa^x-1)(1+t)(1+x^2+aa^{xy}+aa^{xy}x^2)xay.
\end{equation}
and this is 0 if and only if $\GEN{aa^{xy},x^2}$ contains either $aa^x$ or $aa^xt$.
However $aa^x\ne 1, a^xa^{xy}\neq 1$ and $aa^xx^2\neq1$ (for otherwise $a=a^x$, $a=a^y$ or $(ax)^2=1$). Hence
\begin{equation}\label{cd5}
 \text{In cases (c) and (d) either } aa^y=x^2, t=aa^x, t=aa^xx^2, t=aa^y \text{ or } t=aa^yx^2.
\end{equation}


Now we deal separately with the four cases (a)-(d).

{\bf Case (a)} Suppose $aa^{xy}t=1$. Then the last option of (\ref{abc3}) can be excluded because $(xy)^2=t=aa^{xy}$ and hence $(ayx)^2=1$, a contradiction. If $x^2y^2=a^xa^y$ then since $a^xa^y=aa^{xy}=t=x^2y^2(xy)^2$, it follows that $(xy)^2=1$, a contradiction.
If $x^2=aa^y$ or $y^2=a^xa^y$.
In the first case
$$[[a,x+x^{-1}],[a,xy+(xy)^{-1}]] =
4(y^2-1)(1+aa^y+aa^x y^2 + a^x a^y y^2)y = 0,$$
 and thus $y^2\in\{1, aa^y, aa^xy^2,a^xa^yy^2\}$ yielding in all cases to a contradiction.
Finally if or $y^2=a^xa^y=aa^{xy}$ then
$$
[[a,x+x^{-1}],[a,xy+(xy)^{-1}]] =
4(1-a^x a^y)(1+aa^y +aa^y x^2 + x^2)y^{-1}= 0,$$
and therefore $t=y^2=a^x a^y \in \{1,aa^y,aa^y x^2,x^2\}$, a contradiction because $t\ne 1$, $a\ne a^x$, $aa^xx^2 = (ax)^2 \ne 1$ and $x^2\ne y^2$.

{\bf Case (b)} Assume $t=x^2$, or equivalently $x^y=x^{-1}$.

We consider separately the four cases of (\ref{abc3}).
If $aa^y=x^2$ then
$$
[[a,x+x^{-1}],[a,axy+(axy)^{-1}]] =
4(a+a^y+a^yy^2+a^{xy}y^2-a^x -ay^2-aa^{x}- a^y y^2)y=0
$$
and thus $a\in \{a^x,ay^2, aa^x,a^yy^2\}$, a contradiction because $a\ne a^y,a^{xy}$, $y^2\ne 1$ and $1\ne (ay)^2 = aa^y y^2$.
If $x^2=y^2$  then
$$\matriz{{l}
[[a,ax+x^{-1}a],[a,ay+y^{-1}a]] =
4(1+aa^{xy}+aa^xt+a a^y-aa^x-aa^y-t-aa^{xy}t)xy=0 ,
}$$
 and therefore $1\in \{aa^x,aa^y,t, aa^{xy}t\}$, a contradiction because $aa^x,aa^y,t\ne 1$ and $1\ne (axy)^2 = aa^{xy}t$.
If $y^2 = aa^{xy}x^2$ then the same commutator
$$\matriz{{l}
[[a,x+x^{-1}],[a,xy+y^{-1}x^{-1}]] =
4(y^2+x^2y^2+aa^{x}+aa^xx^2-1-x^2-aa^y-aa^yx^2)y =0,
}$$
and hence $1\in \{y^2,x^2y^2,aa^{x}, aa^xx^2\}$. Notice that  since $y^2,x^2y^2,aa^x \ne 1$, it follows that $aa^xx^2=1$. Then $y^2=aa^{xy}x^2=aa^{xy}aa^x=aa^y$ and therefore $(ay)2^1$, a contradiction.

Finally if $y^2=a^xa^y=aa^{xy}=ay^{-1}x^{-1}axy=ayxayxy^2$, then $1=(ayx)^2$ a contradiction.

{\bf Case (c)} Suppose $aa^x=y^2$. In this case the third and fifth options of (\ref{cd5}) take the forms $t=y^2x^2$ and $t=aa^yx^2=a^xy^2a^yx^2=aa^{xy}x^2y^2$, respectively, which implies $(xy)^2=1$ and $(axy)^2=1$, respectively.
This is contradictory with our hypothesis on $x$ and $y$.
The second and fourth options take the forms $t=y^2$ and $t=aa^y=aa^{xy}y^2$ which are cases (b') and (e') respectively.
Since these cases have been excluded we are left with only one case: $aa^y=x^2$. Then
%
%
$$
[[a,(ax)+(ax)^{-1}],[a,(ay)+(ay)^{-1}]] = 
4 (1+aa^{xy}+aa^xt+aa^yt-aa^x-aa^y-t-aa^{xy}t)xy=0 
$$
and thus $1\in \{aa^x, aa^y, t, aa^{xy}t\}$. Recall that $aa^x,aa^y,t\ne 1$ therefore $1=aa^{xy}t$ a contradiction because $1 \ne (xy)^2 = tx^2y^2=aa^{xy}t$.

{\bf Case (d)} Finally suppose $aa^x=x^2y^2$.
Then the second and fourth option of (\ref{cd5}) take the forms $t=x^2y^2$ and $t=aa^{xy}x^2y^2$, respectively and this implies $(xy)^2=1$ and $(axy)^2=1$, respectively, again a contradiction. The third and fifth options take the form $t=y^2$ and $t=aa^{xy}y^2$ which are cases (b') and (e'), already excluded. Thus the only remaining case is $aa^y=x^2$, which is case (c'), already excluded. This finishes the proof of the proposition.
\end{proof}

\medskip
\noindent \textbf{4.2} \underline{Properties of $B$}
\medskip

We now address the properties of $B=\GEN{g\in G: \circ (g)\ne 4}$. For that we start with the following lemma.

\begin{lem}\label{Orden4NoConmutan}
Let $x,y\in G$ with $x^4\ne 1 \ne y^4$ and $(x,y)\ne 1$. Then
\begin{enumerate}
 \item\label{x4NoDihedral} $x^y\ne x^{-1}$ and $y^x\ne y^{-1}$.
 \item\label{x4x2y2} $x^2y^2\ne 1$ and $x^2\ne y^2$.
 \item\label{x4NoDihedraly^2} $(y^2)^x \ne y^{-2}$ and $(x^2)^y\ne x^{-2}$.
\end{enumerate}
\end{lem}

\begin{proof}
The assumption implies that $xy^i,y^ix,x^iy,yx^i\not\in \z(G)$ for every $i\in \Z$.
Hence $(xy^i)^2$, $(y^ix)^2$, $(x^iy)^2$ and $(yx^i)^2$ are all different from 1, by Lemma~\ref{ACentral}.

(\ref{x4NoDihedral}) By symmetry it is enough to show that $x^y\ne x^{-1}$. Otherwise
\begin{eqnarray}\label{eq}
0&=& [[y+y^{-1},xy+(xy)^{-1}],[y+y^{-1},xy^2+(xy^2)^{-1}]] \nonumber\\
&=&2(4y^{-1}+2x^{-2}y+2x^2y+6y^{-3}+3x^{-2}y^3+3x^2y^3+2y^{-5}+x^{-2}y^5+x^2y^5 \\
&& -4y-2x^{-2}y^{-1}-2x^2y^{-1}-6y^3-3x^{-2}y^{-3}-3x^2y^{-3}-2y^5-x^{-2}y^{-5}-x^2y^{-5})\nonumber
\end{eqnarray}

Then $y\in\{y^{-1}, x^{-2}y, x^2y, y^{-3}, x^{-2}y^3, x^2y^3, y^{-5}, x^{-2}y^5, x^2y^5\}$. Having in mind that $y^4\ne 1 \ne x^4$ we deduce that $y^6=1$ or $x^2\in \{y^2,y^{-2},y^4,y^{-4}\}$. If $y^6=1$, then introducing introducing this relation in (\ref{eq}) we get that $x^2\in \{y^2,y^{-2},y^4,y^{-4}\}$ or a contradiction with $y^4\neq 1\neq x^4$. Therefore $x^2=(x^2)^y = x^{-2}$, a contradiction.

(\ref{x4x2y2}) Observe that the inequalities $x^2y^2\ne 1$ and $x^2 \ne y^2$ transfer to each other by replacing $y$ by $y^{-1}$.
Thus it is enough to prove the first inequality. So assume $x^2y^2=1$. Then $(x,y^2)=(x^2,y)=1$ and
\begin{equation}\label{corch1}
\begin{array}{l}
[[x+x^{-1},y+y^{-1}],[x+x^{-1},x^{-1}y+(x^{-1}y)^{-1}]] = \\
2(4x^{-1}+6x^{-1}y^2+2y^{-1}xy+2xyx^{-1}y^{-1}x+y^{-1}xy^{-1}x^2+xy^{-1}xy^{-1}x\\
+3xyxy^{-1}x+3yxy^{-1}x^2+2x^{-3}y^2 \\
-4x-6y^{-2}x-2yx^{-1}y^{-1}-3yx^{-1}y-yx^{-3}y-2xy^{-1}x^{-1}yx^{-1}\\
-3x^{-1}yx^{-1}yx-x^{-1}yx^{-1}yx^{-1}-2y^{-1}x^2y^{-1}x)=0
\end{array}
\end{equation}

Then $x\in\{x^{-1}, x^{-1}y^2, y^{-1}xy, xyx^{-1}y^{-1}x, y^{-1}xy^{-1}x^2, xy^{-1}xy^{-1}x, xyxy^{-1}x, yxy^{-1}x^2, x^{-3}y^2\}$
Having in mind that $x^4\ne 1$, $y^2=x^{-2}\ne x^2$, $(x,y)\ne 1$, $(xy^{-1})^2\ne 1$ and $x^y\ne x^{-1}$, it follows that the only possibility is $x=x^{-3}y^2$.
 However $x$ has coefficient $-4$ while $x^3y^2$ has coefficient $2$, thus expression (\ref{corch1}) is non-zero, yielding to a contradiction.

(\ref{x4NoDihedraly^2}) By symmetry it is enough to prove that $(y^2)^x \neq y^{-2}$. Assume that $(y^2)^x = y^{-2}$. Therefore $(x,y^2)\neq 1$. Let $y_1=xy^2$. Then $(x,y_1)\ne 1$ and if $(y^2)^x=y^{-2}$ it follows that $y_1^2=x^2$. Then $y_1^4\ne 1$ and this contradicts (\ref{x4x2y2}), when applied to $x$ and $y_1$.
%
\end{proof}

\begin{lem}\label{x4No1Conmutan}
$B$ is abelian.
\end{lem}

\begin{proof}
By means of contradiction, let $x,y\in G$ with $x^4\ne 1 \ne y^4$ and $(x,y)\ne 1$.
Once more recall that we are assuming that $RG^+$ is Lie metabelian and $\breve{G}$ is not commutative.
In particular, $xy^i,y^ix,x^iy,yx^i\not\in \z(G)$ for every $i\in \Z$. Hence $(xy^i)^2,(y^ix)^2,(x^iy)^2,(yx^i)^2\ne 1$, by Lemma~\ref{ACentral}.
We consider the following equality where the right column should not be read for the moment.

\begin{equation}\label{xyxyy}
\matriz{{ll}
[[x+x^{-1},y+y^{-1}],[xy+(xy)^{-1},y+y^{-1}]] = &\\
y^{-3}                                                                           & (y^4=1 ) \\
+3y^{-1}+xy^{-2}x^{-1}y+yxy^{-2}x^{-1}                                             & (y^2=1 ) \\
+2x^{-2}y+yx^{-2}                                                                 & (x^2=1 ) \\
+xyx^{-1}+x^{-1}y^{-1}xy^2+y^{-1}xyx^{-1}y+yx^{-1}y^{-1}xy+yxyx^{-1}y^{-1}        & ((x,y)=1 ) \\
+xyxy^2+2yxyxy                                                                    & ((xy)^2=1 ) \\
+2xy^{-1}x+y^{-1}xy^{-1}xy                                                        & ((xy^{-1})^2=1 ) \\
+x^{-1}y^{-2}x^{-1}y^{-1}+2y^{-1}x^{-1}y^{-2}x^{-1}                               & ((xy^2)^2=1 ) \\
+xy^{-1}xy^2+2y^{-1}xyxy+yxy^{-1}xy+3xyx+yxyxy^{-1}                               & (x^y=x^{-1} ) \\
+x^{-1}yxy^2+yx^{-1}yxy+yxyx^{-1}y                                                   & (y^x=y^{-1} ) \\
+2x^{-2}y^{-1}+3y^{-1}x^{-2}+x^{-1}y^{-2}x^{-1}y+yx^{-1}y^{-2}x^{-1}+y^{-2}x^{-2}y   & (x^2y^2=1 ) \\
+xy^{-2}x^{-1}y^{-1}+y^{-1}x^{-1}y^{-2}x+y^{-1}xy^{-2}x^{-1}                         & ((y^2)^x=y^{-2} ) \\
+y^{-1}x^{-1}y^2x                                                                    & ((x,y^2)=1 ) \\
+y^{-2}x^{-2}y^{-1}                                                                  & (x^2y^4=1) \\
+y^{-1}xyx^{-1}y^{-1}                                                                & (y^{x^{-1}}=y^3 ) \\
+xy^3x^{-1}                                                                          & (y^x=y^3 ) \\
+y^{-1}xyxy^{-1}                                                                     & (xyx=y^3) \\
+xy^3x                                                                               & (xy^3x=y ) \\
+y^{-1}x^{-1}y^2x^{-1}                                                               & (x^{y^2}=x^{-1}) \\
- 3y-3x^{-1}y^{-1}x^{-1}-2x^{-1}yx^{-1}-x^2y^{-1}-3x^2y-xy^{-1}x^{-1}\\
-2y^{-1}x^2-2yx^2-y^3-x^{-1}y^{-2}xy -x^{-1}y^{-3}x^{-1}-x^{-1}y^2xy\\
-xy^{-2}xy-xy^{-3}x^{-1}-xy^2x^{-1}y^{-1} -xy^2x^{-1}y -xy^2xy^{-1}\\
-2xy^2xy -2y^{-1}x^{-1}y^{-1}x^{-1}y^{-1}-2y^{-1}x^{-1}y^{-1}x^{-1}y\\
-y^{-1}x^{-1}y^{-1}xy^{-1} -y^{-1}x^{-1}yx^{-1}y^{-1}-y^{-1}x^{-1}yx^{-1}y\\
-y^{-1}x^{-1}yxy^{-1}-y^{-1}x^2y^2-y^{-1}xy^{-1}x^{-1}y^{-1}-y^{-1}xy^{-1}x^{-1}y\\
-y^{-1}xy^2x^{-1} -y^{-1}xy^2x -y^{-2}x^{-1}y^{-1}x^{-1}-y^{-2}x^{-1}y^{-1}x\\
-y^{-2}x^{-1}yx^{-1}-y^{-2}x^{-1}yx-yx^{-1}y^{-1}x^{-1}y^{-1}-yx^{-1}y^{-1}x^{-1}y \\
-yx^2y^2-yxy^{-1}x^{-1}y^{-1}-yxy^{-1}x^{-1}y -yxy^2x^{-1}-yxy^2x
}
\end{equation}
As we are assuming that $RG^+$ is Lie metabelian the expression in (\ref{xyxyy}) should be 0, hence as $y$ appears with coefficient $-3$, one of the elements with positive coefficient should be equal to $y$. Each relation in the right column is equivalent to the one given by each of the summands in the same line to be equal to $y$. Thus one of the relations in the right column of (\ref{xyxyy}) holds.
We will prove that each of these relation yields some contradiction. This is clear for the first seven relations by the first paragraph of the proof.
For the next five relations, it is a consequence of Lemma~\ref{Orden4NoConmutan}. Before continuing with the remaining relations we prove the following claim which will exclude the next two relations.

{\bf Claim}. $(x,y^2)\ne 1$ and $(x^2,y)\ne 1$. By symmetry it is enough to deduce a contradiction from the assumption $(x,y^2)=1$. In this case (\ref{xyxyy}) reduces to
\begin{equation}\label{xyxyyCond}
\matriz{{ll}
[[x+x^{-1},y+y^{-1}],[xy+(xy)^{-1},y+y^{-1}]] = &\\
4y^{-3}                                                                           & (y^4=1 ) \\
+4y^{-1}                                             & (y^2=1 ) \\
+2x^{-2}y+2yx^{-2}                                                                 & (x^2=1 ) \\
+xyx^{-1}+x^{-1}yx+y^{-1}xyx^{-1}y+yx^{-1}y^{-1}xy        & ((x,y)=1 ) \\
+2xy^3x+2yxyxy                                                                    & ((xy)^2=1 ) \\
+2xy^{-1}x+2y^{-1}xy^{-1}xy                                                        & ((xy^{-1})^2=1 ) \\
+2x^{-2}y^{-3}+2y^{-3}x^{-2}                               & ((xy^2)^2=1 ) \\
+4xyx+4y^{-1}xyxy                               & (x^y=x^{-1} ) \\
+x^{-1}y^3x+yx^{-1}yxy+yxyx^{-1}y                                                   & (y^x=y^{-1} ) \\
+4x^{-2}y^{-1}+4y^{-1}x^{-2}   & (x^2y^2=1 ) \\
+xy^3x^{-1}                                                                          & (y^x=y^3 ) \\
- 4y-4y^3-2x^2y^{-1}-4x^2y-2y^{-1}x^2-4yx^2 -2x^2y^3-2y^3x^2\\
-4x^{-1}y^{-1}x^{-1}-2x^{-1}yx^{-1}-xy^{-1}x^{-1} -x^{-1}y^{-1}x \\
- 2 x^{-1}y^{-3}x^{-1}-xy^{-3}x^{-1} -x^{-1}y^{-3}x \\
 -2y^{-1}x^{-1}y^{-1}x^{-1}y^{-1}-y^{-1}x^{-1}y^{-1}xy^{-1} -y^{-1}xy^{-1}x^{-1}y^{-1} \\
-4y^{-1}x^{-1}y^{-1}x^{-1}y -y^{-1}x^{-1}yxy^{-1}-y^{-1}xy^{-1}x^{-1}y\\
-2y^{-1}x^{-1}yx^{-1}y
}
\end{equation}
This is 0 and hence one of the conditions on the right column of (\ref{xyxyyCond}) holds.
The first seven relations are excluded by the first paragraph of the proof. The following three relations are excluded by Lemma \ref{Orden4NoConmutan}. Moreover, if $y^x=y^3$ then $y^2=(y^2)^x = y^6$, a contradiction that finishes the proof of the claim.

So only the last five positive summands of (\ref{xyxyy}) can cancel the $-3y$ and hence at least three of the following conditions hold:
$y^{x^{-1}}=y^3$, $y^x=y^3$, $xyx=y^3$, $xy^3x=y$ and $x^{y^2}=x^{-1}$.
Any two of the first three equalities cannot hold simultaneously because otherwise $(x^2,y)=1$, in contradiction with the Claim.
Then the last two equalities hold. Then $y=xy^3x=y^2x^{-1}yx$ and hence $y^x=y^{-1}$, in contradiction with Lemma~\ref{Orden4NoConmutan}.
\end{proof}

\medskip
\noindent\textbf{4.3} \underline{The exponent of $G$}
\medskip

We will consider separately the cases when $G$ has exponent 4 or different of 4.

\begin{lem}\label{abelianindex2}
If $\Exp(G)\ne 4$ then $[G:B]=2$ and for every $x\in G\setminus B$ and $b \in B$ we have $b^x=b^{-1}$.
\end{lem}

\begin{proof}
Recall that we are assuming that $RG^+$ is Lie metabelian and $\breve G$ is not commutative. Assume that $\Exp(G)\ne 4$.
First we will prove that if for every $x\in G\setminus B$ and $b \in B$ we have $b^x=b^{-1}$ then the index of $B$ in $G$ is equal to 2, or equivalently that $xy\in B$ for every $y\in G\setminus B$. Otherwise we take $b\in B$ with $b^2\ne 1$ then
$b^{-1}=b^{xy}=(b^x)^y=(b^{-1})^y=b$ a contradiction.

Therefore it remains to prove that for every $x\in G\setminus B$ and $b \in B$ we have $b^x=b^{-1}$. By means of contradiction assume that $b^x\neq b^{-1}$.
As $B$ is abelian, it is enough to prove the result for $b\in G$ with $b^4\ne 1$.
Note that $x^4=(bx)^4=(b^{-1}x)^4=1$. Therefore $x^2,(bx)^2\in \z(G)$, by Lemma~\ref{ACentral}, and $(xb)^2 = ((bx)^2)^x=(bx)^2=(bx)^{-2}$.
Hence $1=(bx)^2(xb)^2=bx^{-1}b^2xb$ and thus $(b^2)^x=b^{-2}$.
By assumption
$$\matriz{{lll}
0&=&[[b+b^{-1},x+x^{-1}],[b+b^{-1},x^b+(x^{-1})^b]]  \\
&&+8b^2+8b^{-1}xbx+8b^{-1}x^{-1}bx+8b^2x^2+4b^4+4b^4x^2+4b^{-1}xb^3x+4b^{-1}x^{-1}b^3x \\
&&-8b^{-2}-8bxb^{-1}x-8bx^{-1}b^{-1}x-8b^{-2}x^2-4b^{-4}-4b^{-4}x^2-4xb^{-3}xb-4xb^{-3}x^{-1}b
}$$
Having in mind that $b^4\ne 1$, $(b^{-1}x)^2\ne 1$ and $b^x\ne b^{-1}$, for the $b^2$ to be canceled by the summands with negative coefficient either $x^2=b^4$ or at least two of the following conditions holds: $b^6=1$, $b^6=x^2$, $b^x=b^{-3}x^2$ or $b^x=b^{-3}$. However  the first two equalities are not compatible and the last two are also not compatible. Therefore $b^6\in \{1,x^2\}$ and $b^x \in \{b^{-3},b^{-3}x^2\}$. Thus $b=b^{x^2}=b^{-9}$ and therefore $b^{10}=1$. Then $b^{-4}=b^6\in \{1,x^2\}$. Since $b^4\neq 1$, we conclude that $x^2=b^4$.
Then
$$[[b+b^{-1},bx+(bx)^{-1}],[b+b^{-1},xb+(xb)^{-1}]] =
16(b^{2}
+xb^{-1}xb
-b^6-bxb^{-1}x^{-1})
=0.$$
Thus $b^2=b^{6}$ or $b^{x}=b^{-1}$ yielding in both cases to a contradiction, that finishes the proof of the lemma.
\end{proof}

\begin{lem}\label{Exp4CentroA}
If $\Exp(G)=4$ then $\z(G)=A$.
\end{lem}

\begin{proof}
By Lemma~\ref{ACentral}, $A\subseteq \z(G)$ and we have to prove that the equality holds. By means of contradiction assume that $z\in \z(G)\setminus A$. As $R G^+$ is not commutative and the elements of $G$ of order $2$ are central, there are $x,y \in G$ such that $[x+x^{-1},y+y^{-1}]\ne 0$.
In particular, $t=(x,y)\ne 1$, $x^y\ne x^{-1}$ and $y^x\ne y^{-1}$ and hence $t\not\in \{x^2,y^2\}$.
As $G/A$ has exponent 2, we have $t\in A$ and, in particular $t$ has order 2.
Moreover $z,x,y,xy\not\in A$ and therefore they all have order 4.
Then
$$\matriz{{lll}
0&=&[[x+x^{-1},y+y^{-1}],[x+x^{-1},xyz+(xyz)^{-1}]] \\
&&
8(xz+x^3z+xy^2z+txz^3+x^3y^2z+tx^3z^3+txy^2z^3+tx^3y^2z^3 \\
&&-txz-xz^3-tx^3z-txy^2z-x^3z^3-xy^2z^3-tx^3y^2z-x^3y^2z^3)
}$$
Comparing $xz$ with the terms with negative coefficient, and having in mind that $x,y,xy$ and $z$ have order 4
and $t\not\in \{x^2,y^2\}$, we have that  $z^2\in \{x^2, y^2, x^2y^2\}$.
By symmetry it is enough to consider the cases $x^2=z^2$ and $x^2y^2=z^2$.

Case 1. Assume that $z^2=x^2$. Then
$$\matriz{{lll}
0&=&[[y+y^{-1},xy+(xy)^{-1}],[y+y^{-1},xz]]
=8(t-1)(1+y^2+tx^2+tx^2y^2)yz
}$$
and thus $t\in \GEN{y^2,tx^2}$. However $t\not\in \{1,y^2,tx^2\}$ and therefore $x^2=y^2=z^2$.
Then
$$\matriz{{rcl}
[[xz,yz],[xz,xyz+(xyz)^{-1}]]
=4(1+tx^2-t-x^2)x=0}$$
and hence either $t= 1$ or $x^2=1$, a contradiction.

Case 2. Assume that $z^2=x^2y^2$. Then
$$
[[x+x^{-1},xy+(xy)^{-1}],[x+x^{-1},yz+(yz)^{-1}]]
=16(1+x^2+ty^2+tz^2-t-y^2-tx^2-z^2)xz= 0,
$$
and hence $1\in\{t, y^2, tx^2, z^2\}$ yielding in all cases to a contradiction.

\end{proof}

\begin{lem}\label{Babelianindex2}
Assume $G$ has exponent 4 and let $x,y,h\in G\setminus A$ with $\GEN{x,y,h}$ non-abelian, $x^2\neq y^2$ and $(x,y)=1$.
Then $x^h=x^{-1}$ and $y^h=y^{-1}$.
\end{lem}

\begin{proof}
Let $H=\GEN{x,y,h}$.
If $(xh)^2=1$ then $xh\in \z(G)$, by Lemma~\ref{ACentral}, and thus $(x,h)=1=(y,h)$.
Then $H$ is abelian in contradiction with the hypothesis.
Thus $(xh)^2\ne 1$ and similarly $(yh)^2\neq 1$ and $(x^{\pm 1} y^{\pm 1} h)^2\neq 1$.
As $H$ is not abelian either $(x,h)\ne 1$ or $(y,h)\ne 1$ and by symmetry one may assume that $(x,h)\ne 1$.
Then
$[x-x^{-1},h-h^{-1}]=xh+x^{-1}h^{-1}+hx^{-1}+h^{-1}x-xh^{-1}-x^{-1}h-hx-h^{-1}x^{-1}\ne 0$, since $xh\not \in \{xh^{-1},x^{-1}h,hx,h^{-1}x^{-1}\}$.
This proves that $\breve{H}$ is not commutative. As by assumption $RG^+$ is Lie metabelian, so is $RH^+$ and hence $\z(H)=\GEN{g\in H : g^2=1}$, by Lemma~\ref{Exp4CentroA}. In particular, $y\not\in \z(H)$ and therefore $(y,h)\ne 1$. Moreover $(xy)^2=x^2y^2\ne 1$ and therefore $xy\not\in \z(H)$. Thus $h^x\ne h^y$. For future use we display the information gathered in this paragraph:
	\begin{equation}\label{Babelianindex21}
	 (xh)^2\ne 1, (yh)^2 \ne 1, (x^{\pm 1} y^{\pm 1}h)^2 \ne 1, (x,h)\ne 1, (y,h) \ne 1, h^x\ne h^y.
	\end{equation}

By means of contradiction we assume that either $x^h\ne x^{-1}$ or $y^h\ne y^{-1}$.

\textbf{Claim 1}. $h^x \ne h^{-1}$ and $h^y \ne h^{-1}$.

By symmetry we only prove the second inequality. By means of contradiction assume that $h^y=h^{-1}$. Then $h^x\ne h^{-1}$, by (\ref{Babelianindex21}).
Consider
\begin{eqnarray*}
0&=&[[x+x^{-1},h+h^{-1}],[x+x^{-1},xyh+(xyh)^{-1}]] \\
&=&
8(xy+xyh^2+yx^{-1}+hxh^{-1}y^{-1}+yx^{-1}h^2+hxhy^{-1}+y^{-1}x^2hxh^{-1}+y^{-1}hx^{-1}h \\
&&-yh^{-1}xh-xy^{-1}-yhxh-yx^2h^{-1}xh-xy^{-1}h^2-y^{-1}x^{-1}-yx^2hxh-y^{-1}x^{-1}h^2)
\end{eqnarray*}
As $(x,h)\ne 1$, $y^2\ne 1$, $h^x\ne h^{-1}$, $x^2\ne y^2$ and $(xh)^2\ne 1$, we deduce that either $x^h=x^{-1}$, $h^2=y^2$ or $h^2=x^2y^2$. We consider these three cases separately.

If $x^h=x^{-1}$ then by the initial  assumption $y^h\ne y^{-1}$. Thus $h^2=(y,h)\ne y^2$ and therefore
$$\matriz{{l}
[[y+y^{-1},yh+(yh)^{-1}],[y+y^{-1},xyh+(xyh)^{-1}]] = \\
\hspace{1cm} 16(x+  x y^2+ x^{-1} h^2+ x^{-1} y^2 h^2
-x^{-1}-  x h^2- x^{-1} y^2- x y^2 h^2)= 0,
}$$
and thus $x\in\{x^{-1}, x h^2, x^{-1} y^2, x y^2 h^2\}$, yielding to a contradiction in all cases.

If $y^2=h^2$ then
\begin{eqnarray*}
0&=&[[x+x^{-1},xh+(xh)^{-1}],[x+x^{-1},xyh+(xyh)^{-1}]] \\
&=&16(y+x^2y+x^{-1}h^{-1}xyh+hx^{-1}h^{-1}y^{-1}x^{-1}-y^{-1}x^{-1}hxh-y^3-y^{-1}xhxh-y^{-1}x^2)
\end{eqnarray*}
and thus $y\in\{y^{-1}x^{-1}hxh, y^3, y^{-1}xhxh, y^{-1}x^2\}$.
As $x^2\ne y^2\ne 1$ and $(x,h)\ne 1$, we deduce that $y$ can only be canceled with $y^{-1}(xh)^2$.
Thus $h^2=y^2=(xh)^2$ or equivalently $x^h=x^{-1}$, a case which has been excluded in the previous paragraph.

Finally, assume that $h^2=x^2y^2$. Then
\begin{eqnarray*}
0&=&[[x+x^{-1},xh+(xh)^{-1}],[x+x^{-1},xyh+(xyh)^{-1}]] \\
&=&16(y+y^{-1}h^2+yx^{-1}hxh+yhxhx
-y^{-1}xhxh-yh^2-y^{-1}xhx^{-1}h-y^{-1})
\end{eqnarray*}
Thus $y\in\{y^{-1}xhxh, yh^2, y^{-1}xhx^{-1}h, y^{-1}\}$ and therefore either $y^2=(xh)^2$ or $y^2=xhx^{-1}h$. In the first case $(x,h)=x^2h^2(xh)^2=y^2(xh)^2=1$, contradicting (\ref{Babelianindex21}).
In the second case $x^h = h^{-1} y^2 h^{-1} x = x^{-1}$, a case excluded above. This finishes the proof of Claim 1.

\textbf{Claim 2}. $h^2\not \in \{x^2,y^2,x^2y^2\}$.

Observe that $x,y$ and $h_1=xh$ satisfy the assumptions of the lemma and therefore $h^{-1}x^{-1}=h_1^{-1}\ne h_1^x = hx$, by Claim 1. Hence $h^2\ne x^2$. Similarly $h^2\ne y^2$ and applying this to $x,xy$ and $h$ we deduce that $h^2\ne(xy)^2= x^2y^2$. This proves Claim 2.

\textbf{Claim 3}. $x^h\ne x^{-1}$ and $y^h\ne y^{-1}$.

By symmetry it is enough to prove one of the two conditions and by means of contradiction assume that $x^h=x^{-1}$. Then $y^h\ne y^{-1}$ by the initial  assumption.
Therefore

$$\matriz{{l}
\hspace{-1cm}[[y+y^{-1},h+h^{-1}],[y+y^{-1},xh+(xh)^{-1}]] = \\
=8(
x
+xy^2
+xh^2
+x^{-1}(yh)^2y^2h^2
+xy^2h^2
+x^{-1}(yh)^2h^2
+x^{-1}(yh)^2y^2
+x^{-1}(yh)^2\\
-x^{-1}
-x(yh)^2y^2h^2
-x^{-1}y^2
-x^{-1}h^2
-xyhyh^{-1}
-xy^{-1}hyh

-x^{-1}y^2h^2
-x(yh)^2) \ne 0
}$$
by (\ref{Babelianindex21}) and Claims 1 and 2, a contradiction.

\textbf{Claim 4}. $y^h\ne x^2y$ and $x^h\ne y^2x$.

Again, by symmetry it is enough to prove that the first inequality holds and by means of contradiction we assume that $y^h=x^2y$. Then
$$\matriz{{l}
\hspace{-1cm}[[x+x^{-1},h+h^{-1}],[x+x^{-1},xyh+(xyh)^{-1}]] =\\
8(xy+x^{-1}y+xyh^2+y^{-1}hxh^{-1}+x^{-1}yh^2+y^{-1}hx^{-1}h^{-1}+y^{-1}hx^{-1}h +y^{-1}hxh \\
-yh^{-1}xh-xy^{-1}-x^2yh^{-1}xh-yhxh-y^{-1}x^{-1}-xy^{-1}h^2-xy(xh)^2-y^{-1}x^{-1}h^2) \ne 0
}$$
by (\ref{Babelianindex21}) and Claims 1, 2 and 3, a contradiction.

Finally we consider
\begin{eqnarray*}
0&=&[[x+x^{-1},h+h^{-1}],[x+x^{-1},yh+(yh)^{-1}]] \\
&=&4(y
+x^2y
+yh^2
+hxyh^{-1}x^{-1}
+hxh^{-1}y^{-1}x^{-1}
+hy^{-1}h^{-1}
+yx^2h^2
+xhxyh^{-1}
+xhx^{-1}yh\\
&&
+xhxh^{-1}y^{-1}
+xhx^{-1}hy^{-1}
+hy^{-1}h^{-1}x^2
+xhy^{-1}hx^{-1}
+xhxyh
+hxhy^{-1}x
+hy^{-1}hx^2\\
&&
-yxhx^{-1}h^{-1}
-hyh^{-1}
-y^{-1}
-yhxh^{-1}x
-xyhx^{-1}h
-hyh^{-1}x^2
-hyh
-y^{-1}x^{2}
-y^{-1}h^{2}\\
&&
-hy^{-1}xh^{-1}x^{-1}
-yhxhx
-hyhx^2
-y^{-1}x^{2}h^{2}
-xhy^{-1}xh^{-1}
-hy^{-1}xhx^{-1}
-hy^{-1}xhx).
\end{eqnarray*}
Taking into account the inequalities in (\ref{Babelianindex21}) and Claims 1-4, in order to cancel $y$ we deduce that
$y\in \{hy^{-1}xh^{-1}x^{-1},hyhx^2,xhy^{-1}xh^{-1},hy^{-1}xhx^{-1}\}$.
However, applying Claim 4 to $y,y^{-1}x$ and $h$ we deduce that $(y^{-1}x)^h \ne yx$; applying Claim 2 to $x,y$ and $yh$ we deduce that $(yh)^2\ne x^2y^2$; applying Claim 3 to $x,y^{-1}x$ and $h$ we have $(y^{-1}x)^h\ne x^{-1}y$; and applying Claim 1 to $x,y$ and $yh$ we deduce that $(yh)^{x}\ne h\inv y\inv$. This yields to a contradiction and finishes the proof of the lemma.
\end{proof}

We now introduce a third subgroup of $G$:
	$$C=\GEN{xy : x^2\ne 1 \ne y^2, (x,y)=1}$$

\begin{lem}\label{Expo4Final}
If $\Exp(G)=4$ then
\begin{enumerate}
 \item $\z(G)\subseteq C$,
 \item $C$ is abelian,
 \item $c^t=c^{-1}$ for every $c\in C$ and every $t\in G\setminus C$, and
 \item either $[G:C]=2$ or $C=\z(G)$ and $[G:C]=4$.
\end{enumerate}

\end{lem}

\begin{proof}
(1) As $\Exp(G)=4$ there is an element $x\in G$ of order 4. If $y\in \z(G)$ then $y^2=1$, by Lemma~\ref{Exp4CentroA} and therefore $(xy)^2\ne 1$ and $(x,y)=1$. Therefore $y=x(xy)\inv \in C$.

(2) By means of contradiction we assume that $C$ is not abelian. Then let $x,y,t,u\in G$ such that $1\not\in \{x^2,y^2,t^2,u^2\}$, $(x,y)=(t,u)=1$ and $(xy,tu)\ne 1$. In particular $xy,tu\not\in \z(G)$ and hence $(xy)^2\ne 1 \ne (tu)^2$. Thus $x^2\ne y^2$ and $t^2\ne u^2$. Moreover either $\GEN{x,y,t}$ or $\GEN{x,y,u}$ is not abelian. If both are non-abelian then $x^t=x^{-1}=x^u$, $y^t=y^u=y\inv$, by Lemma~\ref{Babelianindex2}. Then $(x,tu)=(y,tu)=1$, contradicting $(xy,tu)\ne 1$. Thus, by symmetry one may assume that $\GEN{x,y,t}$ is non-abelian and $\GEN{x,y,u}$ is abelian. Then $x^t=x^{-1}$ and $y^t=y^{-1}$. Applying Lemma~\ref{Babelianindex2} to $\GEN{t,u,x}$ we deduce that $u=u^x=u^{-1}$, a contradiction.

(3) Let $t\in G\setminus C$. It is enough to show that if $x^2\ne 1 \ne y^2$ and $(x,y)=1$ then $(xy)^t=(xy)\inv$. If $x^2=y^2$ then $xy\in \z(G)$ and hence $(xy)^t=xy=(xy)\inv$. Otherwise $x^t=x^{-1}$ and $y^t=y^{-1}$, by Lemma~\ref{Babelianindex2}. Therefore $(xy)^t=x^{-1}y^{-1}=(xy)^{-1}$, as desired.

(4) Suppose that $C\ne \z(G)$ and let $c\in C\setminus \z(G)$. If $x,y\in G\setminus C$ with $xy\not\in C$ then by Lemma~\ref{Exp4CentroA} and (3) it follows that $c\ne c^{-1}=c^{xy}=c$. Thus, in this case $[G:C]=2$.

Finally suppose that $C=\z(G)$. This implies that if $[\GEN{x,y,\z(G)}:\z(G)]>2$ then $(x,y)\ne 1$ because otherwise $xy\in C=\z(G)$.
By Lemma~\ref{Exp4CentroA}, $G/\z(G)$ is elementary abelian of order $\ge 4$. We have to prove that the order is exactly 4. Otherwise there are $x,y,u\in G$ such that $[\GEN{x,y,u,\z(G)}:\z(G)]=8$.
These $x,y$ and $u$ will be fixed for the rest of the proof.
Then $1\not\in \{(x,y),(x,u),(y,u),(x,yu)=(x,y)(x,u),(xy,u)=(x,u)(y,u),(xu,y)=(x,y)(u,y),(xy,xu)=(x,u)(y,x)(y,u)\}$ and therefore $\GEN{(x,y),(x,u),(y,u)}$ has order 8.

\noindent $\bullet$ \textbf{Claim 1}. If $[\GEN{g,h,\z(G)}:\z(G)]=4$ then $(g,h)\ne g^2h^2 $.

For otherwise $(gh)^2=g^2h^2(g,h)=1$ and hence $gh\in \z(G)$, a contradiction.

\noindent $\bullet$ \textbf{Claim 2}. If $[\GEN{g,h,\z(G)}:\z(G)]=4$ then $(g,h)\ne g^2$.

By symmetry it is enough to prove the claim for $g=y$ and $h=u$. So assume that $(y,u)=y^2$, or equivalently $y^u=y^{-1}$.

Then $1\not\in \{(xu,y)=(x,y)y^2,(xy,u)=(x,u)y^2,(xy,yu)=(x,yu)y^2,(xyu)^2=xuy^{-1}xyu=xux(x,y)u=(x,yu)x^2u^2\}$ and thus
	\begin{equation}\label{y2NoConmutador}
	 y^2\not\in \{(x,y),(x,u),(x,yu)\} \quad \text{and} \quad (x,yu)\ne x^2u^2
	\end{equation}

Before proving Claim 2 we prove some intermediate claims.

\textbf{Claim 2.1}. $(x,y)\ne x^2$.

By means of contradiction assume $x^2=(x,y)$. Then by (\ref{y2NoConmutador}) $x^2\ne y^2$, $1\ne (x,yu)=x^2(x,u)$ and $1\ne (xyu)^2=u^2(x,u)$. Having these relations in mind we get that
\begin{eqnarray*}
0&=&[[x+x^{-1},u+u^{-1}],[x+x^{-1},yu+(yu)^{-1}]] \\
&=&
8y(1+x^2+u^2+x^2u^2+y^2(x,u)+x^2y^2(x,u)+y^2u^2(x,u)+x^2y^2u^2(x,u) \\
&& -(x,u)-y^2-x^2(x,u)-x^2y^2-u^2(x,u)-y^2u^2-x^2u^2(x,u)-x^2y^2u^2),
\end{eqnarray*}
and therefore, we obtain that $u^2\in \{y^2,x^2y^2\}$. However, if $u^2=y^2$ then
$$\matriz{{l}
[[x+x^{-1}, xu+u^{-1}x^{-1}], [x+x^{-1}, xyu+(xyu)^{-1}]] \\
\hspace{1cm}=16y(1+x^2+y^2(x,u)+x^2y^2(x,u)-(x,u)-y^2-x^2(x,u)-x^2y^2)\ne 0
}$$
because $1\not\in \{(x,u), y^2, x^2(x,u), x^2y^2\}$, a contradiction.
If $y^2=x^2u^2$ then
$$\matriz{{l}
[[x+x^{-1}, xu+u^{-1}x^{-1}], [x+x^{-1}, xyu+(xyu)^{-1}]] \\
\hspace{1cm}=16y(1+x^2+(x,u)y^2+u^2(x,u)-(x,u)-y^2-x^2(x,u)-u^2)\ne 0}$$
because $1\not\in\{(x,u),y^2, x^2(x,u), u^2\}$, a contradiction.

\textbf{Claim 2.2}. $(x,u)\ne x^2$.

By means of contradiction assume that $(x,u)= x^2$. Then $1\ne (xy,u)=x^2y^2$, $1\ne (x,yu)=(x,y)x^2$, $1\ne (y,xu)=(x,y)y^2$. Having in mind these relations we have that
\begin{eqnarray*}
0&=&[[x+x^{-1},y+y^{-1}],[x+x^{-1},xyu+(xyu)^{-1}]] \\
&=&
8xu(1+x^2+y^2+(x,y)u^2+x^2y^2+x^2u^2(x,y)+y^2u^2(x,y)+x^2y^2u^2(x,y) \\
&&-(x,y)-u^2-x^2(x,y)-y^2(x,y)-x^2u^2-y^2u^2-x^2y^2(x,y)-x^2y^2u^2)
\end{eqnarray*}
Therefore $u^2\in \{x^2,y^2,x^2y^2\}$. If $u^2=y^2$, then taking $x_1=x$, $y_1=u$ and $u_1=y$, they satisfy $(x_1,y_1)=x_1^2$ and $(y_1,u_1)=y_1^2$, which contradicts Claim 2.1. By symmetry we may also exclude the case $u^2=x^2$. Therefore $u^2=x^2y^2=(xy,u)$. Taking now $x_1=x$, $y_1=u$ and $u_1=xy$, we have $(y_1,u_1)=(u,xy)=x^2y^2=u^2=y_1^2$ and $(x_1,y_1)=(x,u)=x^2$, again in contradiction with Claim 2.1, that finishes the proof of the claim.

\textbf{Claim 2.3}. $(x,yu)\not\in \{u^2,x^2y^2u^2,y^2u^2\}$ and $x^2\ne u^2$.

By means of contradiction assume first that $(x,yu)=u^2$. Let $x_1=y$, $y_1=yu$ and $u_1=x$. Then $(y_1,u_1)=(yu,x)=u^2=u^2y^2(u,y)=(yu)^2=y_1^2$ and $(x_1,y_1)=y^2=x_1^2$, in contradiction with Claim 2.1.

Secondly assume that $(x,yu)=x^2y^2u^2$. Then $(x,y)(x,u)=(x,yu)=x^2y^2u^2$ and taking $x_1=xu$, $y$ and $u$ we have that $x^2u^2(x,u)=x_1^2= (x_1,y)=(x,y)y^2$, contradiction Claim 2.1.

 Thirdly, assume $(x,yu)=u^2y^2$ and consider $x_1=y$, $y_1=yu$ and $u_1=xy$. Then $(y_1,u_1)=(x,y)(x,u)y^2=(x,yu)y^2=u^2=y_1^2$ and $(x_1,y_1)=(y,u)=y^2=x_1^2$, contradicting Claim 2.1.

Finally assume that $x^2=u^2$ and consider $x_1=y$, $y_1=xu$ and $u_1=u$. Then $(y_1,u_1)=(x,u)=y_1^2$ and $(x_1,u_1)=x_1^2$, contradicting Claim 2.2.
This finishes the proof of Claim 2.3.

\textbf{Claim 2.4}. $y^2\ne u^2$.

By means of contradiction suppose that $y^2=u^2=(y,u)$. Then we have that
$$\matriz{{l}
[[x+x^{-1},xy+(xy)^{-1}],[x+x^{-1},xu+(xu)^{-1}]] \\
\hspace{1cm}=
8yu(1+x^2+(x,yu)+(x,y)y^2+(x,yu)x^2+(x,y)x^2u^2+(x,u)x^2u^2+(x,u)y^2 \\
\hspace{1.3cm}-(x,y)-(x,u)-y^2-(x,y)x^2-(x,u)x^2-x^2u^2-(x,yu)y^2-(x,yu)x^2u^2)  =0}$$
and hence $1\in \{(x,y), (x,u), y^2, (x,y)x^2, (x,u)x^2, x^2u^2, (x,yu)y^2, (x,yu)x^2u^2\}$ which contradicts (\ref{y2NoConmutador}) and Claims 2.1, 2.2 and 2.3.

\textbf{Claim 2.5}. $u^2\ne x^2y^2$.

Otherwise assume that $(y,u)=y^2=u^2x^2$. Hence
$$\matriz{{l}
[[x+x^{-1},xy+(xy)^{-1}],[x+x^{-1},xu+(xu)^{-1}]] \\
\hspace{1cm} 8yu(1+x^2+(x,yu)+(x,y)y^2+(x,u)y^2+(x,yu)x^2+(x,y)u^2+(x,u)u^2 \\
\hspace{1.3cm} -(x,y)-(x,u)-y^2-(x,y)x^2 -(x,u)x^2-u^2-(x,yu)x^2u^2-(x,yu)u^2)=0.}$$
Then $1\in \{(x,y), (x,u), y^2, (x,y)x^2, (x,u)x^2, u^2, (x,yu)x^2u^2, (x,yu)u^2\}$
in contradiction with (\ref{y2NoConmutador}) and Claims 2.1, 2.2 and 2.3.

\textbf{Claim 2.6}. $(x,u)\ne x^2y^2$.

Let $x_1=xu$. If $(x,u)=x^2y^2$ then $x_1^2=(x,u)x^2u^2=y^2u^2$, contradicting Claim 2.5. This finishes the proof of Claim 2.6.

We are ready to prove Claim 2. Consider the following double commutator
$$\matriz{{l}
[[x+x^{-1},y+y^{-1}],[x+x^{-1},xu+(xu)^{-1}]] \\
4xyu(
1
+y^2
+x^2
+(x,yu)
+(x,y)u^2
+(u,x)u^2
+(x,yu)x^2y^2
+(x,yu)y^2
+(x,yu)x^2\\
+(x,y)u^2y^2
+(x,y)u^2x^2
+ (x,u)y^2 u^2
+ (x,u)x^2 u^2
+x^2 y^2
+(x,y)x^2 y^2 u^2
+ (x,u)x^2y^2u^2 \\

-(x,y)
-(x,u)
-u^2
-(x,y)y^2
-(x,y)x^2
- (x,u)y^2
-(x,u)x^2
-u^2y^2
-x^2u^2\\
-(x,yu)u^2
-(x,y)x^2 y^2
- (x,u)x^2y^2
-x^2y^2u^2
- (x,yu)y^2 u^2
- (x,yu)x^2u^2
-(x,yu)x^2 y^2 u^2 )
= 0.
}$$
Then 1 is in the support of the negative part which contradicts (\ref{y2NoConmutador}) and Claims 2.1-2.6. This finishes the proof of Claim 2.

\noindent $\bullet$ \textbf{Claim 3}. If $[\GEN{g,h,\z(G)}:\z(G)]=4$ then $g^2\ne h^2 \ne (gh)^2$.

By Claim 2, $(gh)^2=g^2h^2(g,h)\ne h^2$. Then applying this relation to $g_1=gh$ and $h_1=h$ we have $g^2=(g_1h_1)^2\ne h_1^2 = h^2$.

\noindent $\bullet$ \textbf{Claim 4}. If $[\GEN{g,h,k,\z(G)}:\z(G)]=8$ then $(h,k)\ne g^2h^2$.

By symmetry one may assume that $g=x$, $h=y$ and $k=u$ and by means of contradiction we assume that $(y,u)=x^2y^2$.
Then, since $(x,u)(y,u)=(xy,u)\neq 1$, it follows that $(x,u)\ne (y,u)=x^2y^2$ and $(x,yu)\neq (y,u)=x^2y^2$.
Moreover, by Claim 2, $u^2\ne (y,u)=x^2y^2$.
We collect this information for future use:
\begin{equation}\label{Claim4Inicial}
 x^2y^2 \not\in \{u^2, (x,u), (x,yu)\} 
\end{equation}

\textbf{Claim 4.1}. $(x,u)\ne y^2u^2$.

If $(x,u)=y^2u^2$ then
$$\matriz{{l}
[[x+x^{-1},y+y^{-1}],[x+x^{-1},xu+(xu)^{-1}]] \\
\hspace{1cm} 8xyu(1+x^2y^2+x^2+(x,y)y^2u^2+y^2+(x,y)x^2u^2+(x,y)x^2y^2u^2+(x,y)u^2\\
\hspace{1.3cm} -(x,y)-y^2u^2-(x,y)x^2y^2-(x,y)x^2-x^2u^2-x^2y^2u^2-(x,y)y^2-u^2)= 0
}$$
and thus $1\in\{(x,y), y^2u^2, (x,y)x^2y^2, (x,y)x^2, x^2u^2, x^2y^2u^2, (x,y)y^2, u^2\}$ which contradicts
 Claims 1-3. This finishes the proof of Claim 4.1.

\textbf{Claim 4.2}. $(x,yu)\ne y^2$.


Assume that $(x,yu)=y^2$. Since $y^2=(x,yu)=(x,y)(x,u)$ it follows that $(x,u)\neq y^2$.  Combining this with Claims 2 and 3 and (\ref{Claim4Inicial}) we have
$$\matriz{{l}
[[x+x^{-1},y+y^{-1}],[x+x^{-1},xu+(xu)^{-1}]] \\
\hspace{1cm} 8xyu(1+x^2y^2+y^2+(x,y)u^2+(x,y)y^2u^2+x^2+(x,u)x^2u^2+(x,y)x^2u^2 \\
\hspace{1.3cm} -(x,u)y^2-(x,u)-u^2-(x,u)x^2-(x,u)x^2y^2-x^2y^2u^2-y^2u^2-x^2u^2) \ne 0,
}$$
a contradiction. This finishes the proof of Claim 4.2.

We are ready to prove Claim 4.
Applying Claims 1-3, 4.1 and 4.2 and (\ref{Claim4Inicial}) we deduce that
$$\matriz{{l}
[[x+x^{-1},y+y^{-1}],[x+x^{-1},xyu+(xyu)^{-1}]] \\
\hspace{0.3cm} = 4xu(1+(x,u)+x^2y^2+x^2+(xu,y)x^2y^2u^2+(xy,u)+x^2(x,u)+y^2+(x,yu)u^2\\
\hspace{0.5cm} +(xu,y)u^2+(x,y)x^2u^2+y^2(x,u)+(x,yu)x^2y^2u^2+(x,yu)x^2u^2+(x,y)y^2u^2+(x,yu)y^2u^2\\
\hspace{0.5cm}-(x,y)-u^2-(x,yu)-(x,y)x^2y^2-(x,y)x^2-(x,u)u^2-(y,u)u^2-(y,u)y^2u^2-(x,yu)x^2y^2\\
\hspace{0.5cm}-(x,yu)x^2-(x,y)y^2-(xy,u)u^2-(x,u)x^2u^2-y^2u^2-(x,yu)y^2-(x,u)y^2u^2) \ne 0,
}$$






a contradiction that  finishes the proof of Claim 4.

\noindent $\bullet$ \textbf{Claim 5}. If $[\GEN{g,h,k,\z(G)}:\z(G)]=8$ then $(g,h)\ne k^2$.

Let $g_1=kg^{-1}$, $h_1=g$ and $k_1=gk^{-1}h=g_1^{-1}h$. Then $[\GEN{g_1,h_1,k_1,\z(G)}:\z(G)]=8$.
Thus $(g,h)=(h_1,g_1k_1)=(h_1,g_1)(h_1,k_1)=k^2g_1^2h_1^2(h_1,k_1)\ne k^2$, by Claim 4. This proves Claim 5.

\noindent $\bullet$ \textbf{Claim 6}  If $[\GEN{g,h,k,\z(G)}:\z(G)]=8$ then
$(ghk)^2 \ne \{ g^2h^2, g^2h^2k^2\}$.
In fact  $(ghk)^2g^2h^2k^2=(g,h)(g,k)(h,k)\ne 1$, because $[\GEN{g,h,k,\z(G)}:\z(G)]=8$ ,  and thus $(ghk)^2 \ne g^2h^2k^2$.
Now assume that $(ghk)^2 = g^2h^2$. Then $(g,hk)=h^2(hk)^2$ which contradicts Claim 4.

%
%

Finally, using Claims 1-6 we deduce that
$$\matriz{{l}
\hspace{-0.3cm}[[x+x^{-1},y+y^{-1}],[x+x^{-1},u+u^{-1}]] = \\
2yu((x,y)+(x,u)+(y,u)+(x,y)(x,u)(y,u)+(xyu)^2+\\
(x,y)x^2+(x,y)y^2+(x,y)u^2+(x,u)x^2+(x,u)y^2+(x,u)u^2+(y,u)x^2+(y,u)y^2+(y,u)u^2+\\
(x,y)x^2y^2+(x,y)x^2u^2+(x,y)y^2u^2+(x,u)x^2y^2+\\
(x,u)x^2u^2+(x,u)y^2u^2+(y,u)x^2y^2+(y,u)x^2u^2+(y,u)y^2u^2+\\
(xyu)^2x^2y^2+(xyu)^2x^2u^2+(xyu)^2y^2u^2+\\
(xy)^2u^2+(xu)^2y^2+(yu)^2x^2+\\
(xyu)^2x^2+(xyu)^2y^2+(xyu)^2u^2\\
-1-x^2-y^2-u^2-(x,y)(x,u)-(xu,y)-(xy,u)-x^2y^2-x^2u^2-y^2u^2-\\
(x,y)(x,u)x^2-(x,y)(x,u)y^2-(x,y)(x,u)u^2-(xu,y)x^2-(xu,y)y^2-(xu,y)u^2-(xy,u)x^2-\\
(xy,u)y^2-(xu,y)u^2-x^2y^2u^2-(x,y)(x,u)(y,u)x^2y^2-(x,y)(x,u)x^2u^2-(x,y)(x,u)y^2u^2-\\
(x,y)(x,u)x^2y^2-(xu,y)x^2u^2-(xu,y)u^2y^2-(xy,u)x^2y^2-\\
(xy,u)x^2u^2-(xy,u)y^2u^2-(x,y)(x,u)x^2y^2u^2-(xu,y)x^2y^2u^2-(xy,u)x^2y^2u^2) \ne 0,
}$$
yielding the final contradiction.
\end{proof}

We are ready to finish the proof of the necessary condition in Theorem~\ref{Main}. At the beginning of the section we proved that if $\breve{G}$ is commutative then $RG^+$ is Lie metabelian and $G$ satisfies either condition (1) or (2) of Theorem~\ref{Main}. Assume that $RG^+$  is Lie metabelian but $\breve{G}$ is not commutative as it has been assumed throughout this section. If the exponent of $G$ is different from 4 then, by Lemmas~\ref{x4No1Conmutan} and \ref{abelianindex2}, $B=\GEN{g:\circ(g)\ne 4}$ is an abelian subgroup of $G$ of index 2 and if $x\in G\setminus B$ then $x$ has order 4 and $b^x=b^{-1}$ for every $b\in B$.
Thus $B$ satisfies condition (3) of Theorem~\ref{Main}.
Assume that $G$ has exponent 4 and let $C=\GEN{xy : x^2\ne 1 \ne y^2, (x,y)=1}$.
By Lemmas~\ref{Exp4CentroA} and \ref{Expo4Final}, $C$ is an abelian subgroup of $G$ containing $Z(G)=\{g:g^2=1\}$ and either $C$ has index 2 in $G$ or $C=Z(G)$ and $[G:Z(G)]=4$. In the latter case $G$ satisfies condition (4) of Theorem~\ref{Main}. In the former case, if $t\in G\setminus C$ then $t$ has order 4
and $c^t=c^{-1}$. Thus $G$ satisfies condition (3) of Theorem~\ref{Main}, and the proof finishes.

\end{document}